\documentclass[12pt, leqno]{amsart}

\usepackage{amssymb}
\usepackage{amsmath}
\usepackage{amscd}
\usepackage{float}
\usepackage{graphicx}
\usepackage{amsfonts}
\usepackage{color}
\usepackage{xypic}

\usepackage{amsfonts}
\usepackage[cmtip,arrow]{xy}
\usepackage{cancel}
\usepackage[normalem]{ulem}
\newtheorem{theorem}{Theorem}

\newtheorem{corollary}{Corollary}
\newtheorem{lemma}{Lemma}
\newtheorem{proposition}{Proposition}
\newtheorem{definition}{Definition}
\newtheorem{remark}{Remark}

\def\Z{\Bbb{Z}}
\def\x{\texttt{x}}
\def\F{\mathcal F}

\def\PK{L}

\def\L{\mathcal{L}}
\def\LL{\mathsf{L}}

\def\con{\sqcup}

\def\bra{ \rangle \rangle }
\def\bla{\langle \langle}
 \def\brla{\langle \langle \quad  \rangle \rangle }
\def\J{\mathcal{J}}

\def\tcon{ \ \widetilde{\sqcup} \ }
\def\F{\mathcal F}

\def\L{\mathcal{L}}
\def\K{\mathcal{K}}
\def\W{\mathcal{W}}
\def\f{\mathtt{f}}
\def\OO{\put(8,4){\circle{8}}\quad\,}

\begin{document}

\title{Kauffman type invariants for tied links }

\author{Francesca Aicardi}
 \address{ICTP,  Strada  Costiera  11,   34151  Trieste, Italy.}
 \email{faicardi@ictp.it}
 \author{Jes\'us Juyumaya}
 \address{Instituto de Matem\'{a}ticas, Universidad de Valpara\'{i}so,
 Gran Breta\~{n}a 1111, Valpara\'{i}so, Chile.}
 \email{juyumaya@gmail.com}

\thanks{
The  authors has been supported partially by Fondecyt 1141254}


\date{}

\begin{abstract} We define  two new invariants for tied links. One of
them can be thought  as an extension of the Kauffman polynomial and the
other one as an extension  of the Jones polynomial which is constructed
via a   bracket polynomial for tied links.  These invariants  are   more
powerful than both the Kauffman   and the  bracket polynomials when
evaluated on  classical links.  Further,   the extension of the Kauffman
  polynomial is more powerful of the Homflypt polynomial, as well  as of
certain new invariants introduced recently.  Also we propose a new algebra which
plays  in  the  case of  tied  links the same role as   the  BMW algebra
for  the Kauffman  polynomial  in  the  classical  case.   Moreover,  we
prove  that the  Markov  trace on this  new  algebra  can be recovered
from the extension of  the Kauffman  polynomial defined  here.

\end{abstract}

\maketitle
\section{Introduction}

The tied links constitute a class of knot--like objects, introduced by
the authors in \cite{aijuJKTR1}, which contains the classical links.
The original motivation to introduce these objects arose from the
diagrammatical interpretation of the defining generators  of the
so--called  algebra of braids and ties or simply bt--algebra, see
\cite{aijuICTP1,aijuMMJ1,aijuJKTR1}.

\smallbreak

    Tied links  are no  other than classical  links  whose  set  of
components  is partitioned  into  subsets: two  components  connected  by one or  more   ties belong to  the same  subset  of  the  partition.

A tied link diagram is like the diagram of a link, provided with
ties, depicted as springs connecting pairs of points lying on the curves.
 Classical  links   can  be  considered either  tied  links   with
no  ties  between  different  components (i.e., each  subset of  the
partition contains  one  component)   or tied  links whose components
are  all  tied  together.  Of  course,  classical  knots  coincide  in
 both  cases with  tied  knots.

\smallbreak

In \cite{aijuJKTR1} an invariant for  tied  links, denoted
${\mathcal F}$, is  defined by  skein  relations.   This  invariant
can  be regarded   as  an  extension  of the   Homflypt polynomial,
since it  coincides  with  the Homflypt polynomial  when  evaluated on
knots and  classical links, provided that  they  are considered as
tied  links with all  components  tied  together.

\smallbreak

Notice  that  tied links play an  important role  also  in  the
definition   of  an  invariant  $\Theta$  for  classical  links  which
is  a  generalization of   certain  invariants  derived  from  the
Yokonuma--Hecke algebra  \cite{chljukala}. The  invariant $\Theta$
on links  is  more  powerful  of  the Homflypt polynomial, for details
see \cite[Section 8]{chljukala}.

Also  notice that     the  invariant  $\F$  provides, too,   an
invariant of links more  powerful  than  the  Homflypt  polynomial,
when  evaluated  on  tied  links  without  ties \cite{Aic2}.

\smallbreak

  The invariant  $\F$ can be also constructed through the Jones
recipe \footnote{With this we refer to the mechanism firstly conceived  by V.
 Jones in \cite{joAM} for the construction  of the Homflypt polynomial.}. In fact, we  have proved  that  the  bt--algebra
supports  a  Markov  trace     \cite{aijuMMJ1},  and  that  $\F$  can
be  obtained  as  well  by  the  Jones  recipe applied to the
bt--algebra.  To do this,  we   have introduced the
  algebraic counterpart of  the  braid  group  for tied links, that is the tied  braids
monoid,   and  we have proved  the analogous of   Alexander  and  Markov
theorems  for  tied  links; for details see   \cite{aijuJKTR1}.

\smallbreak

Since  the invariant $\mathcal{F}$ can be  thought as the
Homflypt polynomial for tied links, it is quite  natural  the
question  whether  a generalization  of  the Kauffman  polynomial
  \cite{kaTAMS} can be defined for tied  links.  This paper
proposes and studies  a Kauffman polynomial for unoriented (respectively
oriented)   tied links, denoted by $\L$ (respectively, denoted by
$\widehat{\L}$). We  also define  a sort of Jones polynomial for tied
links,  and  we construct
$\widehat{\L}$ through the Jones recipe applied to a  suitable \lq tied
BMW algebra\rq.  Finally,  using  data from \cite{link},  we
show  pairs  of non--equivalent oriented links which are not
distinguished by  the Hompflypt polynomial, nor by the Kauffman
polynomial, but that are distinguished by $\widehat\L$; moreover the invariant
$\L$ distinguishes  pairs  of oriented links that are not distinguished
by the invariant $\F$ and $\Theta$.

\smallbreak

This  paper is organized as  follows.   Section 2
is dedicated to give the background and notation used in the
paper. In  Section 3,
Theorem \ref{TH1}  proves  the existence of the polynomial $\L$. This
polynomial  is  a three variable invariant obtained by  modifying
suitably  the  Kauffman   skein  relations  \cite[Definition
2.2]{kaTAMS}, that define the Kauffman polynomial $L$ for unoriented
links.    The  polynomial  $\L$ coincides with the polynomial $L$  on
 links  whose  components  are all  tied  together; in particular, the
polynomials $\L$ and $L$  coincide on knots. Moreover,  in
Theorem \ref{TH2} we give the invariant of oriented tied links
$\widehat{\L}$ associated to $\L$. This  is  done by using  the  same     normalization originally  used
  to define the Kauffman polynomial for oriented links \cite[Lemma 2.1]{kaTAMS},
denoted by  $\widehat{L}$, associated to $L$.

\smallbreak

  Section 4 is devoted to define a bracket polynomial $\brla$
for tied links.  In  Proposition \ref{prop2}  we  prove  that
there  exists  a two variable  generalization $\langle\langle \quad
\rangle\rangle$ for  unoriented tied  links of  the one  variable
bracket polynomial \cite{kaTOP}. The polynomial  $\brla$ becomes  the
bracket  polynomial  on  links  whose  components  are all  tied
together,  and  then  coincides  with  the  bracket  polynomial   on
knots. We note  that  $\brla$    results to  be  a  specialization  of
 $\L$. Further,   we obtain a
generalization  of  the  Jones  polynomial   for  tied links,  see
Corollary \ref{cor1}.

\smallbreak

  We start  Section 5  by introducing  a sort of  tied BMW  algebra
with the aim of recovering the invariant $\widehat{\L}$ via the Jones
recipe.  In  Section \ref{sec51}   this tied BMW
algebra,  called  t--BMW algebra,    is defined by generators
and relations; more precisely,   the defining generators    are of  four  types:    usual braid generators, tangle
generators, tied generators and a new class of objects called
tied--tangles generators. The defining relations of the  t--BMW
algebra are  chosen   to  fulfill the same  monomial relations  of the
  bt--algebra \cite{aijuJKTR1, aijuMMJ1},    the monomial relations
of  the  BMW algebra, together   with  a  suitable tied  version of
all  defining  relations  of  the  BMW algebra. In  Section
\ref{sec52} we show that the  diagrammatical  interpretations of  the
defining    generators of  the  t--BMW algebra    agrees with
  the  defining  relations of  it. In Proposition
\ref{finitedimensional}, we  prove  that the  t--BMW  algebra  is
finite  dimensional,   by  showing  that  every element  in it can  be
 put in  a certain {\it  reduced form}. This  result,  together  with
the  existence  of  the  invariant  $\L$, allows  us to  prove  that
the  t--BMW algebra  supports  a Markov  trace,  that  we  denote  by
$\varpi$.  This  trace is   in  fact  similar, cf. (\ref{phi}),  to
   the  Markov  trace   $\tau^{\prime}$  on  the
BMW algebra  by Birman  and Wenzl
\cite{biweTAMS}; thus,    we  obtain     $\widehat{\L}$ by the Jones   recipe. We finish Section 6 by re--proving the fact that
$\widehat{L}$ can be obtained as a \lq specialization\rq\ of
$\widehat{\L}$, see Proposition \ref{kauffanasL}. The proof of this
proposition is completely algebraic and uses the natural  homomorphism
from the t--BMW algebra  onto the BMW algebra and the respective
factorization of the trace $\varpi$ by the trace $\tau^{\prime}$, for
details  see Proposition \ref{tracetBMWwithBMW}.

   \smallbreak

   A motivation to construct   $\L$    was   the
hope  of finding  an invariant    more  powerful  than  the  Kauffman  polynomial
when  calculated   on classical links (tied links without ties); in fact we have got that;  surprisingly,  the
   polynomial  $\brla$, too,     is more powerful  than  the
 Kauffman  polynomial.   In   Section 7   we  give
an  example of  pairs  of  non  isotopic  oriented links distinguished
 by     both normalizations  of   $\L$   and  $\brla$,  which are  not
 distinguished    by  the   Homflypt polynomial  nor  by the  Kauffman
polynomial. Moreover, the invariant $\L$ distinguishes pairs which are
not distinguished  by the invariants  $\F$ and $\Theta$.

\section{  Notations  and  background}

\subsection{}{\it Tied Links.}
Recall a  tied link is a link whose components may be connected by ties.   In fact,  classical  links  form a  subset of  the set of  tied links.   The  ties of  a tied link  define  a  partition  of  the  set  of  its components: two  components  connected by one or  more  ties  belong  to  the  same  set of the partition. Two  tied  links  are  tie--isotopic  if  they  are  ambient  isotopic  and if  the  ties define  the  same  partition  of the  set  of  components.

   A  tie is  depicted  as  a spring  connecting  two points  of  a link.  However,  a tie  is  not a  topological entity: arcs  and  other  ties can   cross  through  it.   The tie--isotopy says that, when remaining in the same equivalence class, it is allowed to move any tie between two components letting its extremes move along the two whole
components;    moreover, ties can be destroyed or created between two components, provided that
these components either coincide, or belong to the same class.
\smallbreak

Two components will be said tied
together, if they belong to the same subset, i.e.,
if between them a tie already exists or a tie can be created.

A  tie  that  cannot  be  destroyed  without  changing  the  tie--isotopy class  is  said  {\it  essential}.  Any  tie connecting  two  points  of  the  same  component  is  not  essential.

Here  we  consider  diagrams of  unoriented  tied  links  in  $S^2$.  Tie--regular  isotopic  tied links diagrams  are  evidently regular  isotopic  links   having  ties    that   define  the  same  partition of  the  set  of  components.
\smallbreak

 Two diagrams of tied  links are said to  be regularly isotopic  if one of them can be carried in the other by using the Reidemeister moves II and/or III, as in the  classical  case.

 \smallbreak

 In  what  follows  the  tie--isotopy  class of  a tied link     will  be  not  distinguished  from   the  class of  its  diagrams.

 \subsection{}
Let  $D$   be  an  unoriented  tied  link  diagram,  and  let $\put(6,4){\circle{8}} \quad$ be  the  zero crossing diagram of the  unknotted circle.
We indicate by $D \con \put(6,4){\circle{8}} \quad $    the disjoint union  of $D$  with  $\put(6,4){\circle{8}} \quad $, and
  by $D \tcon \put(8,4){\circle{8}} \quad $    the  disjoint  union of $D$  with  $\put(6,4){\circle{8}} \quad $,  but  in  this  case  there  is  a  tie between $D$  and  $\put(6,4){\circle{8}} \quad $.

Moreover, we  indicate  by   $D_+$, $D_-$, $D_e$  and  $D_f$   four unoriented tied link  diagrams  that are  identical  except  for  a  disk in  which  they  look, respectively, as  $ \includegraphics[scale=0.5,trim=0 0.4cm 0 0]{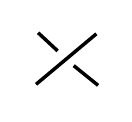}$,   $\includegraphics[scale=0.5,trim=0 0.4cm 0 0]{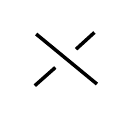}$,
  $\includegraphics[scale=0.5,trim=0 0.4cm 0 0]{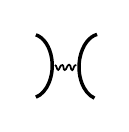}$,  and $\includegraphics[scale=0.5,trim=0 0.4cm 0 0]{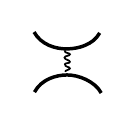} $. The  last two  diagrams without  ties in  the  selected  disc are  indicated  by  $D_0$  and  $D_\infty$.

Let   $\overrightarrow{D}$  be  an oriented  tied link  diagrams.  We  denote  by  $D$ the  unoriented tied  link  associated  to  $\overrightarrow{D}$ and  we  denote  by $w(D)$  the  writhe of $\overrightarrow{D}$.

\subsection{}
  Let now $\overrightarrow{D}$  be  a tied link   whose components are all
tied together. Any tied link    tie--isotopic to  $\overrightarrow{D}$ has the
components all tied, too; hence,  the tie--isotopy of tied links having
all components tied together depends only of the isotopy in $S^2$. Therefore, the
classical links are in topological bijection with the   tied links
having  all components   tied together;    we shall denote   by
 $ \overrightarrow{D}^\nsim$   the classical link  obtained by
forgetting all ties in
  $ \overrightarrow{D}$.

Recall that every tied link can be obtained as the closure of a tied braid, see \cite[Theorem 3.5]{aijuJKTR1}. The set of tied braids with $n$ strands forms a monoid, denoted $TB_n$, which has a  presentation  with  usual braids $\sigma_1,\ldots , \sigma_{n-1}$ and  ties generators $\eta_{1}, \ldots , \eta_{n-1}$ and certain relations, for details see \cite[Definition 3.1]{aijuJKTR1}. The tied monoid $TB_n$ is related to the set of the tied links as the braid group $B_n$   to the classical links, see \cite{aijuJKTR1}. In this last paper was also introduced the  exponent of a tied braid $\theta \in TB_n$, denoted ${\rm exp}(\theta)$. More precisely, if $\theta$  is the product $\theta_1\cdots\theta_k$ in the defining generators of $TB_n$, then
$$
{\rm exp}(\theta) = \sum_{i=1}^{r}k_i,
$$
where $k_i=\pm 1$ if $\theta_i=\sigma_i^{\pm 1}$ and $k_i=0$ if $\theta_i=\eta_i$. Notice that if  $\overrightarrow{D}$ is the oriented tied link obtained by closing the tied braid $\theta$, then
$$
w(D) = {\rm exp}(\theta).
$$

Further, we define $EB_n$ as the subset of $TB_n$ formed by the tied braids of the form:
$$
\eta^n \sigma\quad \text{where}\quad \eta^n:= \eta_1\cdots \eta_{n-1} \quad \text{and}\quad \sigma\in B_n.
$$
The closures of the tied braids of $EB_n$ correspond to    tied links in which the components are all tied together. Notice that
$$
{\rm exp} (\eta^n \sigma) = {\rm exp} (\sigma) \quad \text{where}\quad \sigma\in B_n.
$$
We observe that the monoid $EB_n$ is in fact a group, with identity $\eta^n$, which is naturally isomorphic to the braid group $B_n$; we shall call $\f_n$ this natural group isomorphism,
\begin{equation}\label{f}
\f_n:\eta^n\sigma\mapsto \sigma.
\end{equation}
  Hence,  there  is  a  braid--identification of  the  set  of  tied  links  whose  components  are  all  tied  together  with the  set  of  classical  links.
  Thus, for  $\overrightarrow{D}= \widehat{\eta^n\sigma}$, we have:
$$
\overrightarrow{D}^\nsim= \widehat{\sigma},
$$
  where $\widehat{\sigma}$  is  the  closure  of  $\sigma$.

\subsection{}

Let $K$ be a field. The expression $A$ is a $K$--algebra means that $A$ is an associative  $K$--algebra   with unity equal to $1$;  so, we consider $K$  as  contained in the center of $A$.

  Along  this  paper,  $a, x$ and $z$  indicate  three  commutative   variables.

\section{ Kauffman  polynomial for  tied links}

\subsection{\it  Definition of  $\L$}



\begin{theorem}\label{TH1}  There  exists a unique function
 $$ \L   :  \{ \text{Unoriented  tied  links diagrams } \} \rightarrow \Z[a^{\pm 1}, z^{\pm 1},x^{-1} ]$$
 that  is  defined     by  the  following rules:
 \begin{itemize}
 \item[(i)]  $\L(\put(6,4){\circle{8}} \quad )=1$,

 \item[(ii)] $\L(D \con  \put(6,4){\circle{8}} \quad )= x^{-1} \L(D)$,

 \item[(iii)] $\L$  is  invariant under  Reidemeister  moves  II  and  III,

 \item[(iv)]  $\L (\includegraphics[scale=0.5,trim=0 0.4cm 0 0]{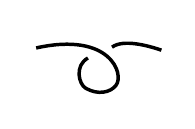})= a  \L (\includegraphics[scale=0.5,trim=0 0  0 0]{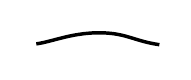})$,

 \item[(v)] $\L (\includegraphics[scale=0.5,trim=0 0.4cm 0 0]{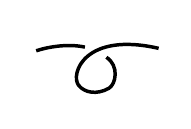})= a^{-1} \L (\includegraphics[scale=0.5 ]{loopno.pdf})$,

 \item[(vi)]  $ \L (\includegraphics[scale=0.5,trim=0 0.4cm 0 0]{crossp.pdf}) +  \L (\includegraphics[scale=0.5,trim=0 0.4cm 0 0]{crossm.pdf})= z\left(
  \L (\includegraphics[scale=0.5,trim=0 0.4cm 0 0]{tanver.pdf}) + \L (\includegraphics[scale=0.5,trim=0 0.4cm 0 0]{tanor.pdf}) \right), $
 \end{itemize}
 where,  as  usual,  $\L (\includegraphics[scale=0.5,trim=0 0.4cm 0 0]{crossp.pdf})$, $\L (\includegraphics[scale=0.5,trim=0 0.4cm 0 0]{crossm.pdf})$,  etc.,
  indicate  the value of    $\L$   on  diagrams  that  are  identical  except  for  a  disc inside  which they  look  like  the  pictures  in parentheses.
\end{theorem}

\begin{remark} \rm The  analogous  of  relation (ii)  in  presence  of  a tie   follows from   (iv), (v) and (vi):
\begin{equation}\label{tcon}   \L( D \tcon \put(8,4){\circle{8}} \quad )= y \L(D), \end{equation}
 where    \begin{equation}\label{y} y:=(a+a^{-1})z^{-1}-1. \end{equation}
\end{remark}

\smallbreak

 Following  exactly  the  same  arguments  as  in  \cite{kaTOP},  we  can  deduce  from  $\L$  a  tie--isotopic   invariant  of  oriented  links,  as the  next  theorem states.

  \begin{theorem}\label{TH2}  Let $\overrightarrow{D}$  be  an oriented   diagram  of  a  tied link.  Then  the  map $\widehat \L$  defined  by

 \begin{equation}\label{Lor}  \overrightarrow{D} \mapsto a^{-w(D)}  \L(D)\end{equation} defines   an  ambient tie--isotopy  invariant  for the tied  links.
 \end{theorem}

\subsection{{\it Some  properties of  $\L$}}
\begin{enumerate}

\item Observe  that  skein  rule  (vi)  is  symmetrical  for  the  exchange  of  the  left  terms (the  respective  diagrams  are also  denoted  by  $D_+$  and  $D_-$),
as well  as for  the  exchange of  the  right  terms. We will denote  the  respective diagrams  of  the right terms  by  $D_{e}$  and  $D_{f}$  (the  subscripts   $e$ and  $f$  are  motivated  form  the  t--BMW algebra, see Section \ref{tBMW})
in order  to  distinguish them  from  the  corresponding  classical  counterpart,  without  ties,  that  are often denoted  respectively  $D_0$  and  $D_\infty$.
Secondly,  observe  that   if  the  two  strands  involved  in  the  crossings of  $D_+$  and  $D_-$ belong  to  two  untied  components,  then  these  component  merge
   in  a  sole  component  of  both diagrams $D_{e}$  and  $D_{f}$.  In  this  case  the ties  in  $D_{e}$  and  $D_{f}$   are   not  essential.  On  the  other  hand,
    if   the  two  strands  involved  in  the  crossings of  $D_+$  and  $D_-$ belong  to  the same  component,  then  they belong  to  two  different components  in
       $D_{e}$  or in $D_{f}$,  that  are  tied  together.
Therefore,  we  have  not  a  skein  rule  involving  $D_0$  and  $D_\infty$,  when  the  strands  belong  to  two untied  components. This  is  the  reason of  the  necessity  of  rule (ii).

\item
We  outline  the  fact  that  our  invariant  $\L=\L(a,z,x)$  is  a  generalization  of   the  Kauffman  invariant  of  unoriented links $\PK=\PK(a,z)$ \cite{kaTAMS}, in  the  sense  that  it is  reduced to  $\PK$   on   classical  unoriented links,  provided  that  they  are  considered  as unoriented  tied links,  whose     components  are  all tied  together; indeed,  observe  that    the classical  Kauffman skein  rule and    relation (vi)  in  this  case coincide (see  also  Section \ref{subsec6}). In  particular,  $\L$  coincides  with  $\PK$ on  knots,  see  examples  in  Section \ref{app}.

\item The symmetry properties  of  $\L$  are  inherited  from  those  of $\PK$.  For  instance,   the   values  of  the polynomial $\L$  on  two  tied link diagrams,  which  are  one  the  mirror  image  of  the  other   coincides  under  the  change of $a$  by  $1/a$.
\end{enumerate}

\begin{proof}[Proof of Theorem \ref{TH1}] The proof is   by  induction on  the  number of crossings and  follows the
 proof  done by Lickorish \cite[page 174]{Lic}  of  the analogous Kauffman's theorem  for classical  links;  we  have just  to pay  attention  to  the points  where  the  presence  of  ties  intervenes  along  the  demonstration.

We suppose  by induction  that  the  function  $\L$  has  been  defined on all the
unoriented links  diagrams  with  at  most  $n$  crossings, i.e.,  $\L$  satisfies  rules (i)--(vi),
provided  that   the Reidemeister moves  do  not  increase  the  number  of  crossings  beyond $n$.

Firstly,  observe  that  for  every  link diagram $D$  with  $n$  crossings,
the  diagrams $D\con  \put(6,4){\circle{8}} \quad$  and
 $D\tcon \put(8,4){\circle{8}} \quad$  have $n$  crossings  and, by  (ii), and  (\ref{tcon}) we  get
 \begin{equation}\label{con} \L(D \con  \put(6,4){\circle{8}} \quad)= x^{-1}\L(D)  \quad \text{and} \quad
   \L( D \tcon \put(8,4){\circle{8}} \quad )= y \L(D). \end{equation}

It  follows  that  if  ${\put(8,4){\circle{8}}\quad\,}^{c,t}$  is  the diagram of  the unlink  with  $c$ unknotted
components without crossings,
and  $t\le c-1$  essential  ties,  then
\begin{equation}\label{Oct}  \L({\put(8,4){\circle{8}}\quad\,}^{c,t})=  y^t / x^{c-t-1}. \end{equation}
Secondly, denoting by  ${\put(8,4){\circle{8}}\quad\,}_n^{c,t}$  a  diagram  of  the same  unlink but  having $n$  crossing,
it  follows  from  rules (iii)--(v)  that
\begin{equation}\label{Octn}  \L({\put(8,4){\circle{8}}\quad\,}_n^{c,t})= a^{\overline w} \L({\put(8,4){\circle{8}}\quad\,}^{c,t}), \end{equation}
where $\overline w$  is the    sum  of  the  writhes of  all  components of ${\put(8,4){\circle{8}}\quad\,}_n^{c,t}$.

Now, suppose  that $D$ is the  diagram of   a tied link    with $c$ components,  $t$  essential  ties
   and   $n+1$
crossing.

If  $D$ is of type ${\put(8,4){\circle{8}}\quad\,}_{n+1}^{c,t}$, i.e., $D$ is  a    $n+1$  crossings  diagram of the  unlink with $c$ unknotted
 components and  $t$  essential  ties,  then    we  still   assume  valid   (\ref{Oct}) and  (\ref{Octn})  so  that  the  value  of $\L$ on $D$  is   defined  by

\begin{equation}\label{Ku}   \L(D)= a^{\overline w(D)}y^t / x^{c-t-1}.  \end{equation}

Otherwise,   by  changing some  undercrossings  to overcrossings we construct   an associated ascending diagram $\alpha D$   with  $c$ component
unknotted  and  unlinked,  having $t$  essential  ties and  $n+1$ crossings.   To  do  this,
it  is  necessary to  provide  the  components of $D$ with  an  order, and
then   choose  a  base point and  an  orientation on  each  component.   In  this  way,
an ordered sequence
of $m\le n$  {\it deciding  points} is  defined,  i.e.,  crossings  where  the  diagrams
$D$  and  $\alpha D$  differ.

In order  to compute $\L(D)$,  we start
  from  the first deciding  point, say $\x_1$,  and  use
skein  relation  (vi). Observe  that if the  diagram   $D^{(1)}:=D$  around $\x_1$  coincides  with $D_+$ (or $D_-$)  then
  $D_-$  (resp. $D_+$) has  $m-1$  deciding  points.  Therefore, $\L(D)$  is  written  in
  terms  of the  value of  $\L$  on  a  diagram  with  $n$ crossings  and $m-1$  deciding  points,
  and  of
  $\L(D_{e} )$  and   $\L(D_{f} )$.  These  last  two  values  are  well  defined
   by    induction, since $D_{e}$  and   $D_{f}$  have $n$  crossings.
   Therefore, it  remains  to  calculate  $\L$  on  $D_-$ (or  on  $D_+$).  Call this diagram $D^{(2)}$  and  apply
    relation (vi)  to the  second deciding  point $\x_2$,   and  so  on,  until,  at  the  last
   deciding  point, $\x_m$, $D_-^{(m)}$ (or  $D_+^{(m)}$) coincides  with  $\alpha D$, i.e., is  a  diagram of  type  $ {\put(8,4){\circle{8}}\quad\,}_{n+1}^{c,t}$,  for  which
    $\L$  is  given  by (\ref{Ku}).

Now,  we  have  to  prove  that  $\L(D)$ does not  depend  on  the  construction  of
$\alpha D$, i.e.,  does  not  depend on  the  component order,  component orientation
and choice of base points.  The  proof  for  tied links is  identical to  that   for
classical  links,  since  the  presence  of  ties  does  not prevent  any  step of  the
arguments (see \cite[Theorem 15.5]{Lic}).

Since  by  a  suitable  choice of  the  base  point and  of  the  orientation and of  the  component order,  any  crossing of  $D$  can  be  a deciding crossing,
 it  follows  that   $\L(D)$   satisfies (vi)
for  every  crossing of the  diagram $D$. Thus, it
   remains  to  prove  that  $\L$ satisfies  (iv) and (v) for  every diagram $D$  with  $n+1$ crossings,  and    that  $\L$  is
invariant under     Reidemeister moves II  and  III,  never
involving more  than $n+1$  crossings.

To prove  rules   (iv) and (v), observe that they  are  satisfied   when $D$ is  of  type ${\put(8,4){\circle{8}}\quad\,}^{c,t}_{n+1}$.
Moreover,  for  every  loop  of type
 $\includegraphics[scale=0.5,trim=0 0.4cm 0 0]{loop1.pdf}$
 or   $\includegraphics[scale=0.5,trim=0 0.4cm 0 0]{loop2.pdf}$  present  in  a  component of  $D$,    it is  always
 possible  to  choose the  base  point on  that  component  so  that  the crossing of  the loop
 is  not a deciding point. Therefore  the same loop is  present  also  in $\alpha D$,  and,
 denoting by  $D'$  the diagram  obtained  from  $D$  by  removing the  loop, we get  by  definition  $\L(\alpha D)= a^{\pm 1} \L(\alpha D')$,  comparing  (\ref{Octn}) and with (\ref{Ku}).  Since  the  factor
 $a^{\pm 1}$ persists  along the  calculation of $\L(D)$, it  follows  that  $\L(D)=a^{\pm 1} \L(D')$, i.e.,
 rules (iv)  and  (v) hold  for  diagrams  with $n+1$ crossings.

Consider  now  the  invariance under  the  Reidemeister move   II.  Firstly,  observe  that  if  $D$  is  of  the  type ${\put(8,4){\circle{8}}\quad\,}_{n+1}^{c,t} $,
then  $\L(D)$ is  invariant  under  such  move; indeed,  if  the  two  strands  involved  in  the  move belong  to  different  components,
their crossings  do  not enter  the  calculation  of  (\ref{Ku}).  On  the other  hand,  if  the  strands  belongs to  the
 same  components,  their  crossings give  opposite  contribution  to  the  writhe  of  that  component,  and  then the  crossings  can  be  destroyed without  effect on $\L$.
  If  $D$  is  not  of  type ${\put(8,4){\circle{8}}\quad\,}_{n+1}^{c,t} $, the  proof proceeds as  follows.
Let  $D_1,D_2,D_3$  and $D_4$  be four  diagrams  with  $n+1$  crossings  that  are  identical  except for a disc in  which  they look
 like the  four fragments  shown in  the  first  line  of Fig. \ref{RIIskein}.  By  applying skein relation (vi)    to  the  top crossing  of  $D_1$, we  get:
$$ \L(D_1)+\L(D_2)= z(\L(D_3)+\L(D_4)).$$   Similarly, let $D'_1,D'_2,D'_3$  and $D'_4$ be  four  diagrams differing  from  the previous  diagrams only  in  the  same  disc  in  which
they  look  as  depicted  in  the  second  line of  the same figure. By  applying skein relation (vi)    to  the  bottom  crossing  of  $D'_1$  we  get:
$$ \L(D'_1)+\L(D'_2)= z(\L(D'_3)+\L(D'_4)).$$  Observe  now  that  $D_2=D'_2$,  $D_3=D'_3$ (for  the  mobility property of  the  tie, see \cite{aijuJKTR1})  and $\L(D_4)=\L(D'_4)$
by rules (iv) and (v) just  proven  (after having  moved  the  ties  far  from  the  disc).  Therefore   $\L(D_1)=\L(D'_1)$.   Finally,  we  observe  that  the  base
points can  be  always  chosen  so   that  $\alpha D_1$ (or  $\alpha D_1'$) is equal  to  $D_1$ (resp, to $D'_1$).  Evidently,  for  (\ref{Ku}),
$\L(\alpha D_1)=\L(\alpha D'_1) =  \L(\alpha D'')$, where the  diagram  $D''$  differs  from  $D$  since  has no  crossings in  the  disc.
In  this  way,  the  calculation  of  $\L(D_1)$  (or  $\L(D'_1)$) does not touch   the  concerned  disc. We  deduce then that  $\L(D_1)=\L(D_1')=\L(D'')$.

The proof  the  $\L$  is  invariant  under    Reidemeister  move  III    is   analogous  to the  corresponding proof for  classical  links (see  \cite[Chapter 15]{Lic}),  still remembering  the  mobility  of  the  ties.

 \begin{figure}
 \centering{ \includegraphics{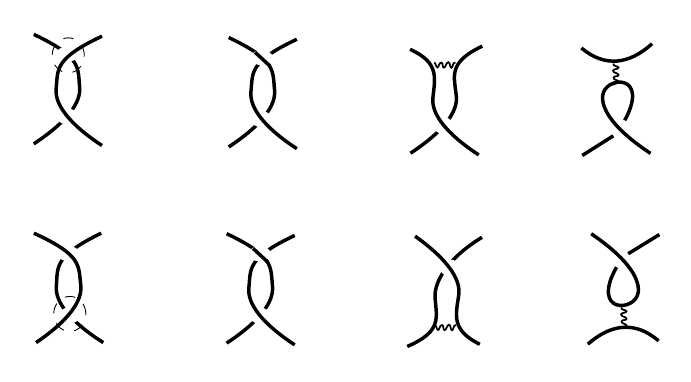} }\caption{Top  line:  Details  of    $D_1,D_2,D_3, D_4$;  bottom line: details  of    $D'_1,D'_2,D'_3, D'_4$. }\label{RIIskein}
 \end{figure}

 \end{proof}

\section{A  bracket  polynomial  for  tied link}

The bracket polynonial is an invariant of regular isotopy for unoriented  links introduced by   Kauffman \cite{kaTOP}. This invariant allows to define the Jones polynomial \cite{joAM} trough a  summation over all state diagrams of the link.   A  similar construction  of the bracket polynomial can  be  done  for tied links, that is, the  definition  is exactly the same  as done by Kauffman if  we  restricts ourselves to crossings of  a  same  component or  between  two  tied components.  For  crossings  of  two  untied  components,  a  new skein relation is  necessary, similar  to  that  used in the  definition  of $\L$.  In this section a Jones  polynomial for  tied  links is  also  proposed.
\smallbreak

Let $A$  and $ c $ be two  commutative independent  variables.
\begin{theorem} \label{prop2}
  There  exists a unique function
$$ \brla  :  \{ \text{Unoriented  tied  links diagrams } \} \rightarrow \Z[A^{\pm 1},c ]$$
   defined  by  the  following  rules:
\begin{enumerate}
\item [(i)] $ \bla \put(6,4){\circle{8}} \quad \bra = 1$,
\item [(ii)] $ \bla \put(6,4){\circle{8}} \quad \con D\bra = c\ \bla \ D \ \bra$,
\item [(iii)] $ \bla \put(6,4){\circle{8}} \quad \tcon D\bra = -(A^2+ A^{-2}) \bla\ D \ \bra $,
\item [(iv)] $\brla$  is  invariant  under  Reidemeister moves  II  and III,
\end{enumerate}
\end{theorem}

  \begin{proof}   If  all  components  are  tied together,  we  can  forget  rules (ii)  and  (vi); further, observe that  the  remaining  rules  coincide  with  the defining  rules  of  the  bracket  polynomial,  where a tie  is  added  between  the   strands.  Then  the  proof   follows    from  \cite[Lemmas 2.3, 2.4 and Proposition 2.5]{kaTOP}.
  Observe also  that,  in  this  case,  item (iv), i.e. the invariance  under  Reidemeister moves II  and  III, can  be  deduced from the other  rules.  In  the absence  of  ties, the invariance  under  Reidemeister  moves II  and  III must  be  stated  and  its consistency     with  the  other  rules must  be  proved.  This  is done  exactly  as  in the    proof  of  Theorem  \ref{TH1}, i.e.,  by  using  skein rule (vi),  where  $z$  is  substituted  by  $(A+A^{-1})$.
\end{proof}

The next proposition, as well as its proof, is analogous to \cite[Proposition
2.5]{kaTOP}.

\begin{proposition}

The  bracket polynomial  $\bla \quad \bra$    satisfies:
\begin{enumerate}
\item[(i)] $\bla \includegraphics[scale=0.5,trim=0 0.4cm 0 0]{loop1.pdf}\bra=
 - A^3\bla \includegraphics[scale=0.5,trim=0 0  0 0]{loopno.pdf}\bra$,
\item[(ii)]
$\bla\includegraphics[scale=0.5,trim=0 0.4cm 0
0]{loop2.pdf}\bra = - A^{-3}\bla \includegraphics[scale=0.5 ]{loopno.pdf}\bra$.
\end{enumerate}
Moreover,  $\bla \quad \bra$ coincides  with  the  Kauffman  bracket  polynomial   on  classical  knots  and on  classical  links,  provided  that  they  are  considered tied  links  whose components   are  all tied  together.
\end{proposition}

\smallbreak
Observe  that  an  ambient  isotopy invariant of oriented  tied  links  can  be  defined in  the  same  way  as  in \cite[Lemma 2.1]{kaTAMS}.
More precisely, we associate to    every oriented link diagram  $\overrightarrow D$, the  polynomial $\J(\overrightarrow D)$  defined by
\begin{equation}\label{J}
\J(\overrightarrow D)  = (-A)^{-3w(D)} \bla D\bra \in {\Bbb Z}[ A^{\pm1}, c^{\pm 1} ]. \end{equation}

\begin{theorem}
The polynomial $\J$   is a  tie--isotopy invariant  for tied links. \end{theorem}

\begin{corollary} \label{cor1} The  polynomial   $\J$, by  the  variable  change  $  A=  t^{-1/4}$,  is
  a  generalization of  the  Jones polynomial for tied links, that  is,  $\J$  becomes  the  one  variable Jones  polynomial on  classical  links,  provided that they  are  considered    tied  links  whose  components   are  all  tied  together. \end{corollary}

\subsection{{\it Relationship  between $\brla$ and  $\L$ }}
Observe  that in  the  same  way  as  in  \cite[Proposition 3.2]{kaTAMS},  we see  that the  polynomial  $\brla$  is  recovered  from  $\L$  by  setting $a:=-A^3$, $z=A+A^{-1}$ and  $x:=c^{-1}$.   That is,   for  every    oriented     tied  link diagram $\overrightarrow D$, we  have
\begin{equation*}
  \J(\overrightarrow D)= (-A)^{-3w(D)}\L(D)_{(a,z,x)=(-A^3, A+A^{-1}, c^{-1})}.
\end{equation*}

\section{The tied BMW algebra}\label{tBMW}
In this section we introduce a sort of  tied BMW algebra, denoted by t--BMW, which  plays  the analogous role for the invariant $\mathcal{L}$, as the BMW algebra \cite{biweTAMS, muOJM} for the Kauffman polynomial $L$.  This  algebra is  defined  by  generators and  relations,   and its defining generators have   diagrammatical  interpretations which are  compatible  with both those  of  the  bt--algebra  and   the BMW--algebra. Also, in this  section  we  prove  that  the t--BMW algebra is  finite  dimensional; this result will  be crucial  to  demonstrate that  the  t--BMW algebra  supports a Markov trace.
\smallbreak

\subsection{}\label{sec51}
  Set   ${\Bbb K}={\Bbb C}(x,z,a)$.  For every  $n\ge 1$,  we  define the tied BMW algebra, or  t--BMW algebra, denoted by ${\mathcal K}_n={\mathcal K}_n(x,z,a )$,   as the ${\Bbb K}$--algebra  equal  to  ${\Bbb K}$  for  $n=1$ and,  for $n>1$, generated by  braids generators
$g_1, \ldots , g_{n-1}$, tangles generators $h_1, \ldots , h_{n-1}$, ties generators $e_1, \ldots , e_{n-1}$ and  tied--tangles generators $f_1, \ldots ,f_{n-1}$  subject to the following relations.


\noindent $A$--braid relations among the $g_i$'s:
\begin{eqnarray}
\label{b1}
g_i g_j & = &  g_j g_i \quad   \text{for}\quad \vert i-j\vert > 1,\\
\label{b2}
g_i g_j g_i &  = & g_j g_i g_j \quad  \text{for}\quad \vert i-j\vert =1.
\end{eqnarray}

\noindent BMW--relations:
\begin{eqnarray}
\label{bmw1}
h_i^ 2 & = & x^{-1}h_i \qquad  \text{ for all }  i , \\
\label{bmw2}
h_ih_j & = & h_jh_i \quad  \text{for}\quad \vert i-j\vert >1 , \\
\label{bmw3}
g_ih_j & = & h_jg_i \quad  \text{for}\quad \vert i-j\vert >1 , \\
\label{bmw4}
g_i h_i & = &  h_i g_i = a^{-1} h_i  \qquad  \text{ for all }  i   , \\
\label{bmw5}
h_i g_jh_i &  = & ah_i\quad  \text{for}\quad \vert i-j\vert =1  , \\
\label{bmw6}
 h_i h_j h_i  &  = &  h_i \qquad \text{for}\quad \vert i-j\vert =1  , \\
 \label{bmw7}
 g_i g_j h_i  &  = &  h_j g_i g_j  \quad = \quad h_jh_i  \qquad \text{for}\quad \vert i-j\vert =1  , \\
 \label{bmw8}
 g_i h_j g_i  &  = &    g_j^{-1}h_i g_j^{-1}    \qquad \text{for}\quad \vert i-j\vert =1  , \\
  \label{bmw9}
 g_i h_j h_i  &  = &    g_j^{-1}h_i     \qquad \text{for}\quad \vert i-j\vert =1  , \\
  \label{bmw10}
 h_i h_j g_i  &  = &   h_i g_j^{-1}     \qquad \text{for}\quad \vert i-j\vert =1.
\end{eqnarray}

\noindent Tied braid relations:
\begin{eqnarray}
\label{bt1}
e_ie_j & =  &  e_j e_i \qquad \text{ for all }  i,j  , \\
\label{bt2}
e_i^2 & = &  e_i \qquad \text{ for all }  i   , \\
\label{bt3}
e_ig_i  & = &   g_i e_i \qquad  \text{ for all }   i     , \\
\label{bt4}
e_ig_j  & = &   g_j e_i \qquad  \text{ for } \vert i - j\vert > 1  , \\
\label{bt5}
e_ie_jg_i & = &  g_ie_ie_j \quad = \quad e_jg_ie_j\qquad \text{ for }   \vert i  -  j\vert =1 , \\
\label{bt6}
e_ig_jg_i & = &   g_jg_i e_j \qquad \text{ for }   \vert i - j\vert =1 .
\end{eqnarray}

\noindent New relations:
\begin{eqnarray}
\label{new1}
 e_i h_i & = & h_i e_i \quad = \quad h_i  \qquad  \text{ for all }    i  , \\
 \label{new2}
 e_i h_j & = & h_j e_i   \qquad  \text{ for }     \vert i - j\vert > 1  , \\
\label{new8}
 f_i e_i & = & e_i f_i \quad = \quad f_i  \qquad  \text{ for all }  i   , \\
 \label{new9}
 f_i e_j & = & e_j f_i  \quad = \quad e_j h_i e_j  \qquad \text{for}\quad \vert i-j\vert =1   , \\
 \label{new10}
f_i g_i  & = & g_i f_i \quad = \quad a^{-1} f_i  \qquad  \text{ for all }  i   , \\
  \label{new13a}
 g_if_j g_i & = & g_j^{-1} f_i g_j^{-1}   \qquad \text{for}\quad \vert i-j\vert =1   ,  \\
 \label{new14}
 g_i+g_i^{-1} & = & z (e_i + f_i) \qquad  \text{ for all }   i.
\end{eqnarray}


The relations  above imply:
\begin{eqnarray}
\label{quadratic}
   g_i^2 & = &  z(e_ig_i+a^{-1}f_i)-1  \text{ for all }   i, \\
 \label{new3}
    f_i^2& = & y \  f_i \qquad  \text{ for  all }   i   , \\
 \label{new5}
 f_i h_i& = & h_i f_i  \quad  = \quad y \  h_i  \qquad  \text{ for all }    i      , \\
   \label{bt7}
e_ig_jg_i^{-1} & = & g_jg_i^{-1}e_j  \qquad \text{ for  $\vert i - j\vert =1$} , \\
\label{new4}
f_i f_j & = & e_j h_i h_j e_i  \qquad \text{for}\quad \vert i-j\vert =1   , \\
\label{new6}
f_i h_j  & = & e_j h_i  h_j \qquad \text{for}\quad \vert i-j\vert =1   , \\
\label{new7}
h_j f_i  & = & h_j h_i  e_j \qquad \text{for}\quad \vert i-j\vert =1   , \\
\label{new11}
 h_i g_j f_i & = & a h_i e_j     \qquad \text{for}\quad \vert i-j\vert =1   , \\
\label{new12}
 f_i g_j h_i & = & a e_j h_i   \qquad \text{for}\quad \vert i-j\vert =1   , \\
 \label{new13}
 f_i g_j f_i & = & a e_j  h_i e_j    \qquad \text{for}\quad \vert i-j\vert =1 ,
\end{eqnarray}
  where   $y:=  (a+a^{-1})z^{-1}-1$. Observe  that  relations (\ref{new3}) and (\ref{new5})  are obtained  by using the expression of $f_i$ given by (\ref{new14}); namely
 $$
f_i^2 = f_i \left(\frac{g_i+g_i^{-1}}{z}-e_i\right)
      = \frac{a^{-1}f_i+ a f_i}{z}- f_i
     = y f_i
$$ follows  from  (\ref{new10}) and  (\ref{new8}), whereas
$$
f_i h_i= \left(\frac{g_i+g_i^{-1}}{z}-e_i\right) h_i
= \frac{a^{-1}h_i+ a h_i}{z}- h_i
= y h_i
$$ follows from  (\ref{bmw4} ) and  (\ref{new1}).

\subsection{\it  Diagrams for  the  t--BMW algebra} \label{sec52}
\mbox{}

The  natural  counterpart of  a  classical  $n$--tangle  in the  context of  tied  links is  a {\it tied $n$--tangle}, i.e.,  a  rectangular  piece of a diagram of a tied link with  $n$  arcs and other closed  curves in generic position,  such  that the     arcs  have $n$  end points at  the top  and  $n$ end points at bottom of  the  rectangle, cf. \cite{mowa}. Ties  connect  the  curves  inside  the  rectangle.  Observe  that,  because of  the   ties' mobility property \cite[6.3.3]{aijuMMJ1}, the  ties can   lie  entirely  inside or outside the  rectangle.

The  relation of  tie--isotopy is  thus  extended   to  tied  tangles according to the  following  definition.

\begin{definition}\label{tangleisotop} Two tied  $n$--tangles  are {\it  regular tie--isotopic} if,  by
 substituting   the  first  one, as a part of  a  tied  link,  with the  second one,   the  regular tie--isotopy  class of  the  tied link  is  preserved.
\end{definition}

  For  our purpose  we need    to consider here  only certain  particular  tied $n$--tangles,  that  we  are  going  to introduce; later we  will define an  algebra of  such  diagrams  that  result  naturally  isomorphic  to   the algebra $\K_n$.


 Observe  that  the defining generators of  the  t--BMW algebra consist of four sets of generators: the $g_i$'s, which  correspond (diagrammatically) to the usual braid generators; the $h_i$'s,  which correspond  to the   tangle generators of  the BMW--algebra; the $e_i$'s,  already  introduced in the bt--algebra, which correspond  to the ties, and finally  certain  news elements $f_i$'s.  This last set of  generators   in fact corresponds to  tangle generators  with a  tie connecting  the up and  bottom arcs.

More precisely, for  every  $n \ge 1$, we  associate to   the unity of  the  algebra $\K_n$  the  trivial  braid diagram made  by $n$  parallel  vertical  threads and for  $n > 1$,
and  to the defining generators of  ${\mathcal K}_n$  with  index $i$  we  associate  diagrams  coinciding with the identity except for  the  $i$--th  and  $(i+1)$--st  strands
as shown in the figure below.

\begin{figure}[H]
 \centering{
\includegraphics[scale=.5]{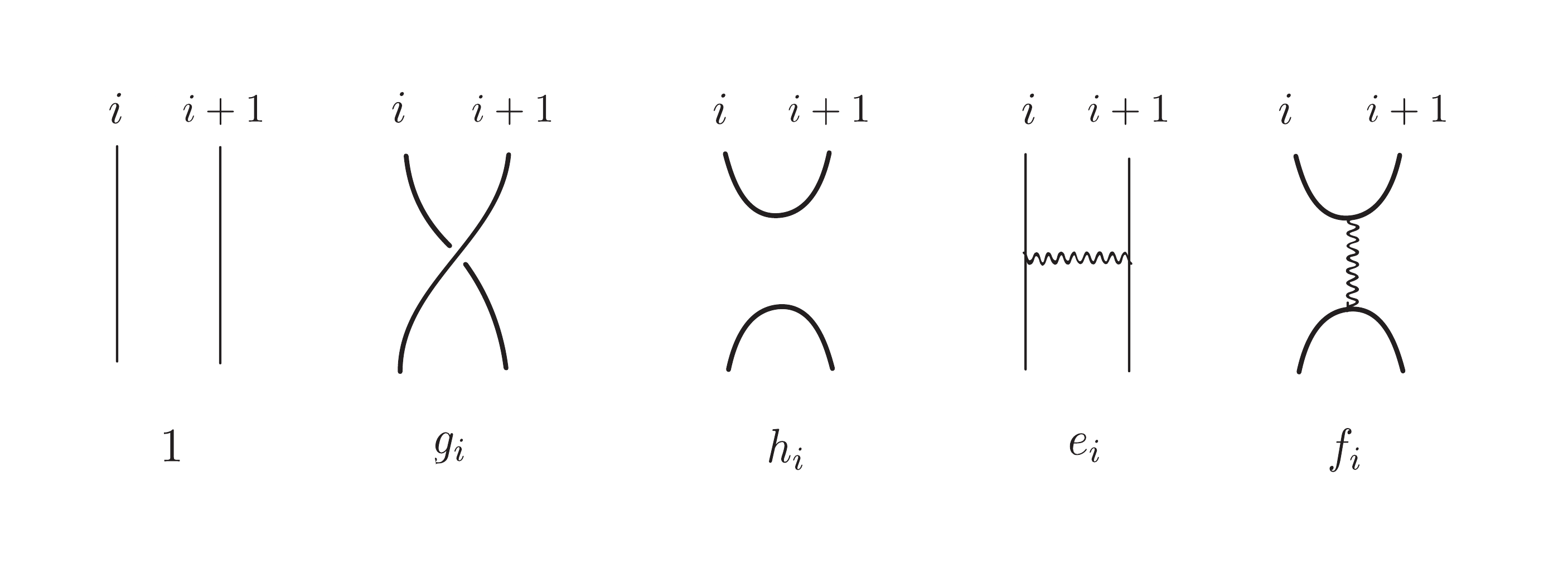}}
\caption{Diagrammatical interpretation of the defining generators of $\mathcal{K}_n$.}\label{fig1}
\end{figure}

 The (associative)  multiplication of  diagrams    is defined   by concatenation, i.e.,   the multiplication   $d_1d_2$   is done by putting the diagram $d_2$ below of the diagram $d_1$.

  Let  $W_n$ be  the  set  of  diagrams obtained by  translating the words in the  generators of   $\K_n$   generator  by  generator. $W_n$ is  provided with  the multiplication by  concatenation.  Denote ${\Bbb K} W_n$ the ${\Bbb K}$--vector space with basis $W_n$  and    extend linearly  the product to ${\Bbb K}W_n$. We define  the algebra $\mathcal{W}_n$ as the  ${\Bbb K}$--algebra constructed from ${\Bbb K}T_n$ by factoring out the defining relations of  $\K_n$.

\begin{figure}[H]
 \centering
\includegraphics[scale=.5]{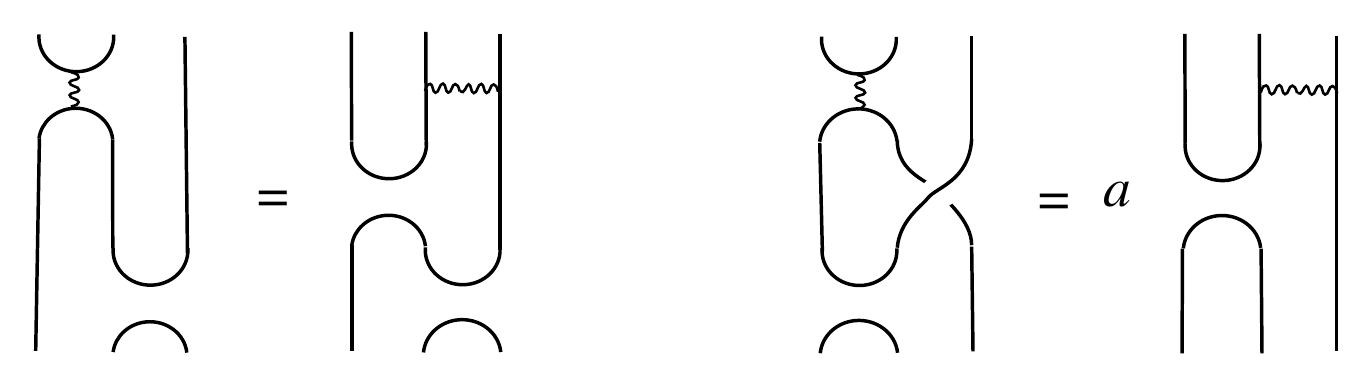}
\caption{Relations (\ref{new6}) and (\ref{new12}). }\label{fig2}
\end{figure}

\begin{remark} \label{trisot} \rm Observe that   every  monomial defining relation   of $\K_n$, when  translated  to  the corresponding pair of diagrams, is a  relation of tie--isotopy of  tied tangles,  according to   Definition \ref{tangleisotop} (see Fig.  \ref{fig2}, left).
Note  that there  are exactly  five  non  monomial   defining  relations,  namely    relations (\ref{bmw1}), (\ref{bmw4}), (\ref{new14})  and the second   relations in (\ref{bmw5}) and (\ref{new10});  an example is shown in Fig.  \ref{fig2}, right.
 \end{remark}

 In particular,   the  above  construction defines     a natural algebra isomorphism
  $\chi_n:\K_n \longrightarrow \mathcal{W}'_n$,   where   $\mathcal{W}'_n$  is  the  image
of $\K_n$ in  $\W_n$.

\begin{remark}\label{tower}  \rm  Observe  that,  for  every word  $\alpha \in K_n$,   the  diagram $\chi_n(\alpha) \in \mathcal{W}'_n$  can be identified  with  the  diagram in $\mathcal{W}'_{n+1}$,  obtained  by  adding at  right of $\chi_n(\alpha)$ one   vertical  thread.  This  defines  a  natural injective  morphism  of  $\mathcal{W}'_n$ in $\mathcal{W}'_{n+1}$, for $n\ge 1$.  We thus  use  the isomorphisms  $\chi_n$'s  to  obtain the tower of
algebras $\K_1\subset\K_2\cdots $,  which is obtained    by identifying,  for  every $n\ge 1$,    $\K_n$ with the subalgebra of $\K_{n+1}$
generated by the defining generators of $\K_{n+1}$:  $g_i$'s,
$e_i$'s, $h_i$'s and $f_i$'s for $1\leq i\leq n-1$.
\end{remark}

\subsection{{\it The finite dimension  of  the t--BMW algebra}}

 In $\mathcal{K}_n$ we define the subset $\Gamma_n$ by
$$
\Gamma_n:=\{1, g_{n-1}, g_{n-1}e_{n-1}, h_{n-1}, e_{n-1}, f_{n-1} \},
$$
and we shall say that   a word  in the defining generators  of $\mathcal{K}_n$ can be {\it reduced}, if it  can be written  as  a linear combination of words having at most one element of $\Gamma_n$.  Further, we say that an element in ${\mathcal K}_n$ can be reduced if it can
be written as   linear combination of reduced words.

\begin{proposition}\label{finitedimensional}
Every element in  ${\mathcal K}_n$  can be reduced. Thus  ${\mathcal K}_n$ is finite dimensional.
\end{proposition}

To prove   this proposition  we need some relations and reductions which are grouped in the following lemmas.

\begin{lemma}\label{ij>1}
For all $i,j$ such that $\vert i - j\vert > 1$  the following relations hold:
\begin{enumerate}
\item[(i)]
$f_i f_j   =  f_j f_i $,
\item[(ii)]
$f_i h_j   =  h_j f_i $,
\item[(iii)]
$f_i g_j  =  g_j f_i  $,
\item[(iv)]
$f_i e_j =  e_j f_i.$
\end{enumerate}
\end{lemma}
\begin{proof}
It is sufficient to write, by (\ref{new14})
\begin{equation}\label{fi=}
 f_i = z^{-1}(g_i + g_i^{-1}) -e_i.
\end{equation}
Then all relations of the lemma follow  from (\ref{b1}), (\ref{bmw3}), (\ref{bt1}), (\ref{bt3}) and (\ref{bt4}).
\end{proof}
\begin{lemma}\label{ij=1}
For all $i,j$ such that $\vert i - j\vert =1$  the following relations hold:
\begin{enumerate}
\item[(i)]
$f_if_jf_i  = e_j h_i e_j$,
\item[(ii)]
$g_ig_jf_i = f_jg_ig_j$,
\item[(iii)]
$f_if_jh_i = e_jh_i$,
\item[(iv)]
$h_if_jf_i = h_ie_j$,
\item[(v)]
$f_if_je_i = f_if_j$,
\item[(vi)]
$e_if_jh_i = f_jh_i $,
  \item[(vii)]
$h_if_jh_i  =  h_i$.
\end{enumerate}
\end{lemma}
\begin{proof}
Relation (i) is obtained as follows:
$$
f_if_jf_i = (e_jh_ih_je_i)f_i = e_jh_ih_jf_i = e_jh_ih_jh_ie_j = e_jh_ie_j,
$$
where the first to the fourth equalities are obtained, respectively, by using  (\ref{new4}), (\ref{new8}), (\ref{new7}) and (\ref{bmw6}).

A proof of (ii) follows by expanding $f_i$ (see (\ref{fi=})) in both sides of the equality and  having in mind (\ref{b2}) and (\ref{bt6}).

From (\ref{new6}) and (\ref{new8})  it  follows that   $f_if_jh_i = f_ie_ih_jh_i = f_ih_jh_i$. Then applying again (\ref{new6}) and later (\ref{bmw6}), we obtain  (iii).

We prove now (iv), we have:
\begin{eqnarray*}
h_i f_jf_i & = & h_ih_je_if_{i} \quad \text{from (\ref{new7})}\\
& = & h_ih_jf_i  \quad \text{from (\ref{new1})}\\
& = & h_ih_jh_ie_j  \quad \text{from (\ref{new7})}\\
& = & h_ie_j \quad \text{from (\ref{bmw6})}.
\end{eqnarray*}

Relation (v) is a direct consequence of (\ref{new9}) and (\ref{new8}).

Relation (vi) is a direct consequence of (\ref{new9}) and (\ref{new1}).

  Relation (vii)  is  obtained by applying (\ref{new7}),  then (\ref{new1}) and (\ref{bmw6}).

\end{proof}

\begin{lemma}\label{reductions1}
  The following relations hold in $\mathcal{K}_n$:
\begin{enumerate}
\item[(i)] $ h_n g_{n-1}g_ne_n =   h_n f_{n-1}$,
\item[(ii)]  $g_ne_ng_{n-1}h_n =  f_{n-1} h_n$.
\end{enumerate}
 \end{lemma}
\begin{proof}
For (i), we have:
\begin{eqnarray*}
h_n g_{n-1}g_ne_n  & = & h_n h_{n-1}  e_n \qquad \text{from} \quad (\ref{bmw7})\\
& = &
 h_n f_{n-1}  \qquad \text{from} \quad (\ref{new7}). \\
  \end{eqnarray*}
 For (ii), we have:
\begin{eqnarray*}
g_n e_n g_{n-1}h_n   & = & e_n g_n g_{n-1}  h_n    \  = \  e_n  h_{n-1} h_n \qquad \text{from} \quad (\ref{bmw7})\\
& = &
 f_{n-1} h_n  \qquad \text{from} \quad (\ref{new6}).\\
  \end{eqnarray*}
\end{proof}

\begin{lemma}\label{reductions2}
The following relations hold in $\mathcal{K}_n$:
\begin{enumerate}
\item[(i)] $ h_n f_{n-1}g_n = h_n f_{n-1}g_n e_n =h_ne_{n-1}g_{n-1}^{-1}$,
\item[(ii)] $g_ne_nf_{n-1}g_n =  g_nf_{n-1}g_ne_n=  g_ne_nf_{n-1}g_ne_n=g_{n-1}^{-1}e_{n-1}h_ne_{n-1}g_{n-1}^{-1} $,
\item[(iii)] $f_nf_{n-1}g_n =f_nf_{n-1}g_ne_n = e_{n-1}h_ne_{n-1}g_{n-1}^{-1}$,
\item[(iv)] $ g_nf_{n-1}f_n = g_ne_nf_{n-1}f_n= g_{n-1}^{-1}e_{n-1}h_ne_{n-1} $,
\item[(v)] $g_nf_{n-1}h_n = g_ne_nf_{n-1}h_n= g_{n-1}^{-1}e_{n-1}h_n$.
\end{enumerate}

\end{lemma}
\begin{proof}
We prove  (i). We have
\begin{eqnarray*}
h_n f_{n-1} g_n & =  & h_nh_{n-1}e_ng_n \quad 	\text{from (\ref{new7})}\\
& = &
h_nh_{n-1}g_ne_n \quad 	\text{from (\ref{bt3})}\\
& = &
h_ng_{n-1}^{-1}e_n \quad 	\text{from (\ref{bmw10})}\\
& = &
h_ne_ng_{n-1}^{-1}e_n \quad 	\text{from (\ref{new1})}\\
& = &
h_ne_ne_{n-1}g_{n-1}^{-1} \quad 	\text{from (\ref{bt5})}\\
& = &
h_ne_{n-1}g_{n-1}^{-1} \quad 	\text{from (\ref{new1})}.
\end{eqnarray*}
In the same way we prove that the second term is equal to the third  term.

We prove (ii):
\begin{eqnarray*}
g_ne_n f_{n-1}g_n & = & g_n e_nh_{n-1}e_n g_n \quad 	\text{from (\ref{new9})}\\
& = &
 g_ne_ne_{n-1}h_{n-1}e_{n-1}e_ng_n \quad 	\text{from (\ref{new1})}\\
& = &
 e_ne_{n-1}g_n h_{n-1}g_ne_n e_{n-1} \quad 	\text{from (\ref{bt5})}\\
 & = &
 e_ne_{n-1}g_{n-1}^{-1}h_ng_{n-1}^{-1}e_n e_{n-1} \quad 	\text{from (\ref{bmw8})}\\
 & = &
g_{n-1}^{-1} e_ne_{n-1}h_ne_n e_{n-1}g_{n-1}^{-1} \quad 	\text{from (\ref{bt5})}\\
& = &
g_{n-1}^{-1} e_{n-1}h_n e_{n-1}g_{n-1}^{-1} \quad 	\text{from (\ref{new1}) and (\ref{bt1}). }
\end{eqnarray*}
In the same way we prove that the second term is equal to the fourth term. By using  (\ref{new9}), (\ref{bt4}) and (\ref{bt2}) we get that the  third term is equal to the first:
$e_ng_n f_{n-1} e_ng_n = e_ng_ne_n f_{n-1}g_n=  e_ng_n f_{n-1}g_n$.

We prove now (iii). We have:
\begin{eqnarray*}
f_nf_{n-1}g_n & = & e_{n-1}h_nh_{n-1}e_ng_n\quad 	\text{from (\ref{new4}) }\\
& = &
 e_{n-1}h_nh_{n-1}g_ne_n\quad 	\text{from (\ref{bt3}) }\\
 & = &
 e_{n-1}h_ng_{n-1}^{-1}e_n\quad 	\text{from (\ref{bmw10}) }\\
  & = &
 e_{n-1}h_ne_ng_{n-1}^{-1}e_n\quad 	\text{from (\ref{new1}) }\\
  & = &
 e_{n-1}h_ne_ne_{n-1}g_{n-1}^{-1}\quad 	\text{from (\ref{bt5}) }\\
 & = &
 e_{n-1}h_ne_{n-1}g_{n-1}^{-1}\quad 	\text{from (\ref{new1}). }
\end{eqnarray*}
I.e., the first term is equal to the third. Now, by using this equality we have $f_nf_{n-1}g_ne_n =e_{n-1}h_ne_{n-1}g_{n-1}^{-1}e_n$; but, by using  (\ref{bt5}) we have $e_{n-1}g_{n-1}^{-1}e_n=e_ne_{n-1}g_{n-1}^{-1} $, hence
$$
f_nf_{n-1}g_ne_n = e_{n-1}h_ne_ne_{n-1}g_{n-1}^{-1} = e_{n-1}h_ne_{n-1}g_{n-1}^{-1},
$$
where the last equality is from (\ref{new1}).

In analogous way we obtain (iv).

For (v), we have
\begin{eqnarray*}
g_n f_{n-1}h_n & = & g_ne_nh_{n-1}h_n \quad \text{from (\ref{new6})}\\
& = & e_ng_nh_{n-1}h_n \quad \text{from (\ref{bt3})}\\
& = & e_ng_{n-1}^{-1}h_n \quad \text{from (\ref{bmw9})}\\
& = & e_ng_{n-1}^{-1}e_nh_n \quad \text{from (\ref{new1})}\\
& = &g_{n-1}^{-1} e_{n-1}e_nh_n \quad \text{from (\ref{bt5}).}
\end{eqnarray*}
Thus, by using now  (\ref{new1}) we obtain $ g_n f_{n-1}h_n=g_{n-1}^{-1} e_{n-1}h_n $.
Finally, this last equality,  (\ref{bt5}) and  (\ref{new1}) implies:
$$
e_ng_n f_{n-1}h_n = e_ng_{n-1}^{-1} e_{n-1}h_n =g_{n-1}^{-1} e_{n-1}e_nh_n  =g_{n-1}^{-1} e_{n-1}h_n.
$$

Relation (\ref{new14}) implies that the words in  (v) can be reduced.
\end{proof}

\begin{lemma}\label{L4}
The following relations hold in $\mathcal{K}_n$:
\begin{enumerate}
\item[(i)] $h_ng_{n-1}e_n = h_n e_{n-1}g_{n-1}$,
\item[(ii)] $e_ng_{n-1}h_n = g_{n-1}e_{n-1}h_n$,
\item[(iii)] $e_n g_{n-1}f_n = g_{n-1}e_{n-1}f_n$,
\item[(iv)] $g_ne_ng_{n-1}e_n = e_{n-1}g_ng_{n-1}e_{n-1} $,
\item[(v)] $g_n e_n g_{n-1}g_ne_n= g_{n-1}g_ne_ne_{n-1}g_{n-1}$,
\item[(vi)] $g_ne_ng_{n-1}f_n = f_{n-1}g_ng_{n-1} e_{n-1}$.
\end{enumerate}

\end{lemma}

\begin{proof}
 By using (\ref{new1}), we have $h_ng_{n-1}e_n = h_ne_ng_{n-1}e_n$, then from (\ref{bt5}), we obtain
$h_ng_{n-1}e_n = h_ne_ne_{n-1}g_{n-1}$, and using again (\ref{new1}), we obtain (i).
Analogously,  we get (ii).

For (iii)  we have:  $e_ng_{n-1}f_n= e_n g_{n-1}e_n f_n =  g_{n-1}e_{n-1}e_n f_n  $ since  (\ref{new8}) and (\ref{bt5}). Hence,  (iii)  follows from    (\ref{new8}).

To  obtain (iv),  we  use  (\ref{bt5}) and  (\ref{bt6}).

For  (v),  in  virtue of  (\ref{bt3}), we deduce
\begin{eqnarray*}
g_ne_ng_{n-1}g_ne_n
  & = &
 g_ng_{n-1}e_ne_{n-1}g_n \qquad \text{from} \quad (\ref{bt6})\\
& = &
 g_ng_{n-1}g_ne_ne_{n-1} \qquad \text{from} \quad (\ref{bt5})\\
& = &
 g_{n-1}g_ng_{n-1}e_ne_{n-1} \qquad \text{from} \quad (\ref{b2}).
 \end{eqnarray*}
 Then (v)  follows  now  from  (\ref{bt5}).

To  obtain (vi),  observe  that relation  (\ref{new10}) implies  $g_ne_ng_{n-1}f_n = a^{-1}g_ne_ng_{n-1}g_nf_n$. By  relations
 (\ref{bt6}) and (\ref{b2}) we get
 $g_ne_ng_{n-1}f_n = a^{-1}g_{n-1}g_ng_{n-1} e_{n-1}f_n$; now   using on this last factor (\ref{new9}), (ii) Lemma \ref{ij=1} and (\ref{new10}), we deduce
  (vi).

\end{proof}

\begin{proof}[Proof of Proposition \ref{finitedimensional}]
The proof is by  induction on $n$. For $n=2$,  the proposition  follows directly from the relations (\ref{bmw1}), (\ref{bmw4}), (\ref{bt2}), (\ref{new1}), (\ref{new3}), (\ref{new5}), (\ref{new8}), (\ref{new10}) and (\ref{quadratic}); thus we  assume now that $n>2$. Every  element in ${\mathcal K}_n$ can be written as a linear combination of elements  in  the form $w=w_1d_1w_2d_2\cdots w_kd_k $, where  $w_i\in \mathcal{K}_{n-1}$ and $d_i\in \Gamma_n$.     Now we  shall  use induction on $k$ and   we see that it is enough to consider  $k=2$.    Hence, we are going to prove that for $w=w_1d_1 w_2d_2$ the proposition holds. By the induction hypothesis we have  that $w_2=  \omega_1d\omega_2$, where $\omega_i\in  {\mathcal K}_{n-2}$ and $d\in \Gamma_{n-1}$; then, $w = w_1\omega_1 d_1dd_2\omega_2$. Thus, to finish the proof of the proposition, we need only to see that $d_1dd_2$ can be reduced in $\mathcal{K}_n$.

Trivially $d_1dd_2$ is reduced if $d_1$ or $d_2$ are equal to  $1$. So, we consider now  $d_i\not=1$.

For the case $d=1$, we use again (\ref{bmw1}), (\ref{bmw4}), (\ref{bt2}), (\ref{new1}), (\ref{new3}), (\ref{new5}), (\ref{new8}), (\ref{new10}) and (\ref{quadratic}), to obtain $d_1dd_2$ in the   reduced form.
As for  the remaining  $125$ cases to reduce,   we observe that the  50 cases   obtained when $d=g_{n-1}$ and $d= f_{n-1}$ are the most  representatives.  We  omit  the  analysis of  the  remaining  cases  since  they can  be obtained  either in an  analogous  manner or directly  from  the  algebra relations.

 In the tables of reduction below   we put in the first line the possibilities of $d_1$ and in the first column the possibilities of $d_2$ in the product $d_1dd_2$ and  we will indicate  how  to reduce  the product $d_1dd_2$.

For the case $d=g_{n-1}$, we have the following reduction table.

{\small

$$
\begin{array}{c|ccccc}
         & g_n  & h_n & e_n &  g_ne_n & f_n  \\
   \hline
g_n     & (\ref{b2})  & (\ref{bmw7}) & (\ref{bt6}) & (\ref{bt6}) \, \text{and} \, (\ref{b2}) & \, \text{(ii) Lemma \ref{ij=1}} \\
h_n     &  (\ref{bmw7}) &  (\ref{bmw5}) & \text{(i) Lemma \ref{L4}} & \text{(i) Lemma \ref{reductions1}} & (\ref{new11}) \\
e_n     &  (\ref{bt6})  & \text{(ii) Lemma \ref{L4}} & (\ref{bt5})&  (\ref{bt6})   & \text{(iii) Lemma \ref{L4}} \\
g_ne_n  &  (\ref{bt6}) \, \text{and} \, (\ref{b2})   & \text{(ii) Lemma \ref{reductions1}}  & \text{(iv) Lemma \ref{L4}}   & \text{(v) Lemma \ref{L4}} & \text{(vi) Lemma \ref{L4}} \\
f_n     & \text{(ii) Lemma \ref{ij=1}}  &  (\ref{new12})& (\ref{new8})  \, \text{and} \, (\ref{bt5})&  \text{(ii) Lemma \ref{ij=1}\, and}\, (\ref{new9})& (\ref{new12}) \\
\end{array}
$$
}

For the case $d=f_{n-1}$, the  reduction table is the following.
{\small
$$
\begin{array}{c|ccccc}
 & g_n  & h_n & e_n &  g_ne_n & f_n  \\
   \hline
g_n     &  (\ref{new13a}) \, \text{and} \, (\ref{new14})    &  \text{(v) Lemma \ref{reductions2}}& (\ref{new9})  &   \text{(ii) Lemma \ref{reductions2}} & \text{(iv) Lemma \ref{reductions2}} \\
h_n    &  \text{(i) Lemma \ref{reductions2}} &  \text{(vii) Lemma \ref{ij=1}  } & (\ref{new9}) \,\text{and}\,(\ref{new1})&\text{(i) Lemma \ref{reductions2}} & \text{(iv) Lemma \ref{ij=1}}\\
e_n     & (\ref{new9})  &  \text{(vi) Lemma \ref{ij=1}} & (\ref{new9}) &  (\ref{new9})  &  (\ref{new9}) \,\text{and}\,(\ref{new8}) \\
g_ne_n  &  \text{(ii) Lemma \ref{reductions2}}  &   \text{ (v) Lemma \ref{reductions2}} & (\ref{new9}) & \text{(ii) Lemma \ref{reductions2}}    & \text{(iv) Lemma \ref{reductions2}} \\
f_n     &  \text{(iii) Lemma \ref{reductions2}}  &     \text{(iii) Lemma \ref{ij=1}  }&  \text{(v) Lemma \ref{ij=1}} & \text{(v) Lemma \ref{ij=1} and }\,(\ref{new10}) & \text{ (i) Lemma \ref{ij=1}}
\end{array}
$$
}
\end{proof}

\section{A trace  for  the t--BMW algebra}\label{sec6}
The  fact, stated in  Proposition \ref{finitedimensional},  that every element in  ${\mathcal K}_n$  can be reduced,
and  the  existence of  the  invariant  $\L$ (Theorem \ref{TH1}), allow us  to  prove   that the  t--BMW algebra  supports  a  Markov trace (see next Theorem \ref{Thtrace}).  That is, we  prove the existence of  an unique
family $\varpi =\{\varpi_n\}_{n\geq 1}$, where  $\varpi_{n+1}: \K_{n+1}\rightarrow {\Bbb K}$ is a linear map defined  from $\varpi_n$ and   its values on $\alpha g_n$, $\alpha e_n$, $\alpha h_n$ and $\alpha f_n$,  for any  $\alpha \in \K_n$.  We  finish  this  section with an  algebraic  proof  that  the Kauffman polynomial $\widehat L$  is  a  specialization of  the  polynomial $\widehat \L$.

\subsection{}

Let $\alpha$  be  a word in $\K_n$. In what  follows  we  call  for short {\it closure of } $\alpha$,  denoted  by  $\widehat \alpha$,  the  closure  of  the  image $\chi_n(\alpha)$ in  ${\mathcal T}_n$.   Observe that $\widehat \alpha$ is the diagram of an unoriented  tied  link.
To prove the  existence of  a Markov  trace  on the t--BMW algebra  we  need  the  following  lemma.

\begin{lemma}\label{lemmaX}  If  the  word $\alpha \in  \K_n$  can be  written,  using  the  algebra  relations,  as  a linear combination  of  words  $\beta_i$, i.e.,  $\alpha = \sum_{i=1}^m k_i \beta_i$,  where   $k_i \in  \mathbb K$, then
$$   \L(\widehat \alpha)= \sum_i k_i \L(\widehat \beta_i). $$
\end{lemma}

\begin{proof}  Observe  that in  $\K_n$ every  splitting  of a  word  in  a  linear  combination  of  words  is  done  by  using  the basic  relations (\ref{b1})--(\ref{new14})  of  the  algebra. Therefore,   we  have  firstly to  prove  that  for every  monomial  relation  in  $\K_n$  of  type
     $ \alpha= \alpha'$,  the  diagrams   $\widehat \alpha$  and   $\widehat \alpha'$  are  regularly isotopic,  so that  the  value  of  the  polynomial $\L$ coincides on  them.  Observe  that  relations  (\ref{bmw1}),(\ref{bmw4}),(\ref{bmw5}) and (\ref{new10}),  are  considered  non monomial  for the  presence  of  a  coefficient  different  from  1.  Secondly, we have to  prove that for  every  non  monomial basic relations  of  type  $\alpha=\sum_i k_i \beta_i$,  we  have  $\L(\widehat \alpha)= \sum_i k_i \L(\widehat \beta_i)$.     The  proof  consists  in  a  verification  relation  by  relation  and  does not entail  any difficulty.
\end{proof}

We define $\varpi_1$ as the identity; for $n>1$,   $\varpi_n$  is given in  the following theorem.

  \begin{theorem}\label{Thtrace}
 The t--BMW algebra supports a  unique  Markov trace $\varpi=\{\varpi_n\}_{n\geq 1}$, where for every positive integer  $n>1$ the  linear map $\varpi_n: \mathcal{K}_n \longrightarrow {\Bbb K}$ is  defined   by the following rules (recall  that by  definition  $y=(a+a^{-1})z^{-1}-1$):

\begin{enumerate}
\item[(i)]$\varpi_n(1)=1$,
\item[(ii)] $\varpi_n( \alpha \ \beta)=\varpi_n (\beta\  \alpha)$,
\item[(iii)] $\varpi_{n+1}(\alpha g_n)=\varpi_{n+1}(\alpha g_ne_n)= (x/a) \ \varpi_{n}(\alpha)$,
\item[(iv)] $\varpi_{n+1}(\alpha h_n)=\varpi_{n}(\alpha f_n)= x \ \varpi_{n}(\alpha)$,
\item[(v)] $\varpi_{n+1}(\alpha e_n)= \ x\ y  \ \varpi_{n}(\alpha) $,
\end{enumerate}
where  $\alpha, \beta \in \K_n$.
\end{theorem}

\begin{proof}
Given  any  $\alpha  \in  \K_n$,  let  $\widehat \alpha$ be  the  diagram  of the unoriented tied link  obtained  as  the  closure of  $\chi_n(\alpha)$. For  every positive integer $n$,  we define $\omega_n$ by:

\begin{equation}\label{phi} \omega_n(\alpha)=  x^{n-1} \L(\widehat \alpha). \end{equation}

We will prove  that  $\omega_n$  satisfies    (i)--(v).  Since for  every  $n$,   $\omega_n$ is  uniquely defined,  we  obtain  that $\varpi_n=\omega_n$ for all  $n$.

We observe  firstly that  when   $\alpha=1 \in \K_n$, the   closure  of  $\alpha$  is  $\OO^n$  and,    according  to  Theorem \ref{TH1}, $\L(\OO^n)=x^{1-n}$.  So,  by the definition of $\omega_n$,  we  get   $\omega_n(1)=1$. Secondly,  the  closures  of  $\alpha \beta$  and  of  $\beta \alpha$  are  diagrams of  the  same  tied link; therefore,  by   (\ref{phi}),  $\omega_n(\alpha \beta)=\omega_n(\beta \alpha)$.  Hence  $\omega_n$  satisfies (i) and (ii).

The  proof   that  $\omega_n$  satisfies  (iii)--(v)   is  done  by  induction on  $n$.  For  $n=1$,  the algebra  $\K_1$   contains only  the  unit  element, i.e.,  the  trivial braid  composed  of  a  sole  thread;  its closure is  the   circle  $\OO^1$,  with  $\L(\OO^1)=1$,  so  that  $\omega_1(1)=1$.

We  now  suppose  that  $\omega_j$  satisfies  (iii)--(v) for  $j\le n$.   Let   $\delta \in  \K_{n+1}$.  By  Proposition  \ref{finitedimensional}, $\delta$ can  be  written  as
$$\delta=  \sum_{i=1}^m k_i \beta_i,$$
where, for  every  $i$,  $\beta_i\in \K_{n+1}$  contains  only  one  element $\gamma_i^{(n+1)}$ of  the  set  $\Gamma_{n+1}$, i.e.
           $$\beta_i=  \beta'_i \gamma^{(n+1)}_i \beta''_i,$$
          where  $\beta'_i ,  \beta''_i  \in  \K_n$.

Observe  that,  by Lemma \ref{lemmaX},  $   \L(\hat \delta)= \sum_{i=1}^m k_i \L(\hat \beta_i) $.
Now, for  every  $i$,  we  consider  the  element  $$\tilde \beta_i= \beta''_i \beta'_i\gamma_i^{(n+1)}.$$  This  element  has  the  same  closure  as  $\beta_i$  and  then  $\omega_{n+1}(\beta_i)= \omega_{n+1}(\tilde \beta_i)$.  Observe that $\tilde \beta_i$ is  the  product  of  an  element in  $\K_n$,  namely  $\beta''_i \beta'_i $,  by an  element  in $\Gamma_{n+1}$.
Therefore,  it is  sufficient to prove that  $\omega_{n+1}$   satisfies (iii)--(v) when  $\delta=\tilde\beta_i$.

Let  $\alpha \in \K_{n}$,  and  $\widehat \alpha$  its  closure,  so  that  \begin{equation*}\L(\widehat \alpha)= \omega_n(\alpha) / x^{n-1}. \end{equation*}

(iii).  The  closures of $\alpha g_{n}$  and  of  $\alpha g_n e_n$ are  different  from  $ \widehat \alpha$  for  the  presence  of  a  new loop,  and possibly of an  unessential tie.  Therefore,  by  Theorem \ref{TH1},
 \begin{equation}\label{T1} \L(\widehat {\alpha g_n})=\L(\widehat {\alpha g_n e_n})= a^{-1} \L(\widehat {\alpha} ) .\end{equation}  By   (\ref{phi})
 $$   \L(\widehat {\alpha g_n })=\omega_{n+1}(\alpha g_n) /x^{n}  \quad \text{and} \quad  \L(\widehat {\alpha g_n e_n}) = \omega_{n+1}(\alpha g_n e_n)/x^n  .$$
  Therefore    (\ref{T1}) implies
 $$ \omega_{n+1}(\alpha g_n)  =\omega_{n+1}(\alpha g_n e_n)= x a^{-1}\omega_n (\alpha) .$$

(iv). The  closures of $\alpha h_{n}$  and  of  $\alpha  f_n$  are  regularly  isotopic  to  the  closure of  $\alpha$, so  $\L(\widehat {\alpha h_n})=\L(\widehat {\alpha f_n})= \L(\widehat {\alpha} )$. By     (\ref{phi})
 $$   \L(\widehat {\alpha h_n})=\omega_{n+1}(\alpha h_n)/ x^{n} \quad \text{and} \quad  \L(\widehat {\alpha f_n})  = \omega_{n+1}(\alpha f_n)/x^n  .$$
 Hence
 $$ \omega_{n+1}(\alpha h_n)  =\omega_{n+1}(\alpha f_n)= x \omega_n (\alpha) .$$

 (v).  The  closure  of  $\alpha e_n$ is  obtained  from  the  closure  of $\alpha$  by  adding  a separated circle,  tied  to  $\widehat \alpha$.  By    (\ref{tcon}), \begin{equation}\label{T2}\L(\widehat {\alpha e_n})= y \L(\widehat {\alpha} ) .\end{equation}  By (\ref{phi}),
 $$   \L(\widehat {\alpha e_n})=\omega_{n+1}(\alpha e_n) x^{n}.  $$
Therefore  from (\ref{T2}) we obtain
$$ \omega_{n+1}(\alpha e_n)   = x y \omega_n (\alpha) .$$


 For  completeness,  we  consider also  the  case  when  $\gamma_i^{(n+1)}$ is  the  identity $1_n \in \K_{n+1}$.  In  this  case  the  closure  of  $\alpha \gamma_i^{(n+1)}$  is  obtained  from  the  closure  of  $\alpha$  by  adding  a  separated  circle.  Therefore,  by  theorem  \ref{TH1},  $\L(\widehat {\alpha 1_n})= \L(\widehat \alpha ) /x$.
 By    (\ref{phi}),  $\omega_{n+1}(\alpha 1_n)= \omega_n (\alpha)$.

\end{proof}

\begin{remark} \rm Observe  that \begin{equation}\label{traceiv} \varpi_{n+1}(\alpha g_n^{-1})=\varpi_{n+1}(\alpha g_n^{-1}e_n)= x \ a \  \varpi_{n}(\alpha).\end{equation}
Indeed, we  have   $$\alpha g_n^{-1}  = -\alpha  g_n + z (\alpha e_n + \alpha f_n),$$
and
$$\alpha g_n^{-1}e_n  = -\alpha  g_ne_n + z (\alpha e_n + \alpha f_n).$$
Now, by using (iii), (iv) and (v)  we  obtain:
 $$\varpi_{n+1}(\alpha g_n^{-1})=  \varpi_{n+1}(\alpha g_n^{-1})e_n  = -x/a \varpi_n(\alpha) + z( x \ y \varpi_n(\alpha) + x \varpi_n(\alpha)= x (-a^{-1} + z(y +1)  ) \varpi(\alpha).$$
   So,  $\varpi_n$  satisfies    (\ref{traceiv}), in virtue of    (\ref{y}).

\end{remark}

\begin{remark}\rm

Having present   (\ref{phi}) and the definition of $\widehat{\L}$, we
deduce that for the oriented  tied link  $\widehat{\alpha}$, with
$\alpha \in TB_n$, we have
$$\widehat{\L} (\widehat{\alpha} ) =
\left(\frac{1}{x}\right)^{n-1}a^{\rm exp(\alpha)}(\varpi_n \circ
\pi_n)(\alpha),$$
where $\pi_n$ is the representation of $TB_n$ defined by mapping
$\sigma_i$ to $g_i$ and $\eta_i$ to $e_i$; hence, in the setting of
the Jones recipe, the t--BMW algebra is the corresponding algebra to
define $\widehat{\L}$. Now, define  the tied Temperley--Lieb algebra,  denoted  t--${\rm TL}_{n}$,
as the subalgebra of $\mathcal{K}_n$ generated by the $h_i$'s, $e_i$'s,
and $f_i$'s. Notice that   this subalgebra in fact can be presented
abstractly through  relations (\ref{bmw1}), (\ref{bmw2}),
(\ref{bmw6}), (\ref{bt1}), (\ref{bt2}) and (\ref{new1})--(\ref{new9}).
Furthermore, note that   a natural  epimorphism from the tied
Temperley--Lieb algebra to the classical  Temperley--Lieb algebra is
simply  obtained   by  mapping   $e_i\mapsto 1$ and $  f_i\mapsto
h_i$. This epimorphism and the original construction of the Jones polynomial suggest  that the invariant ${\mathcal J}$ can  be  constructed through the Jones
recipe applied to  the algebra t--${\rm TL}_{n}$.
\end{remark}

\subsection{}\label{subsec6}

In this subsection we study  the factorization of $\varpi_n$ trough   the trace on the BMW--algebra,  then we re--prove  that  the Kauffman polynomial of oriented  links corresponds to a  specialization of $\widehat \L$.
\smallbreak

Let $\ell$ and $m$ be two commutative variables. The BMW algebra ${\mathcal C}_n = {\mathcal C}_n( \ell , m)$ was introduced  by J. Birman and H. Wenzl \cite[Section 2]{biweTAMS} and independently by J. Murakami \cite{muOJM}. This algebra is defined through  a presentation
with braid generators, tangles generators and  relations among them that are motivated by topological reasons.
We consider here the reduced presentation of this algebra  defined in \cite[Definition 1]{coetalJA}; more precisely, ${\mathcal C}_n$ is the ${\Bbb C}(\ell , m)$--algebra
 defined through braid generators $G_1, \ldots ,G_{n-1}$ and tangle generators $H_1, \ldots ,H_{n-1}$, subject to braid relations among the $G_i$'s and the following relations:
\begin{eqnarray}
G_iH_i &  = & \ell^{-1}H_i \qquad \text{for all $i$}, \\
H_iG_jH_i & = & \ell H_i  \qquad \text{for}\quad \vert i-j\vert =1, \\
G_i + G_i^{-1} & = & m(1 + H_i) \qquad \text{for all $i$}.
\end{eqnarray}

In \cite[Theorem 3.2]{biweTAMS}, cf. \cite[Subsection 9.6]{joCBMS}, is proved that   the family  $\{{\mathcal C}_n\}_{n\in {\Bbb Z}_{>0}}$ supports  a Markov trace   $\tau'=\{\tau_n^{\prime}\}_{n\in {\Bbb Z}_{>0}}$,  where the $\tau_n^{\prime}$'s are   linear maps defined uniquely by the axioms: $\tau_n^{\prime}(1) = 1$ and for all $c, d\in {\mathcal C}_n$, we have
$\tau_n^{\prime}(c \, d) = \tau_n^{\prime}(d \, c)$ and
\begin{eqnarray}\label{traceBMW}
\tau_{n+1}^{\prime}(c \, G_n) &= & \frac{m}{\ell(\ell + \ell^{-1} -m)}\tau_n^{\prime}(c),\\
\tau_{n+1}^{\prime}(c \, H_n)& =&\frac{m}{\ell + \ell^{-1} -m}\tau_n^{\prime}(c).
\end{eqnarray}
This Markov trace allows to define the Kauffman polynomial,  $\widehat{L}$, in terms purely algebraic. More precisely, for the link  $\overrightarrow{D}$ obtained as the closure  of  $\sigma\in B_n$, we have:
\begin{equation}\label{kauffmanpolynomial}
\widehat{L}_{m,l} (\overrightarrow{D}) =\left(\frac{\ell + \ell^{-1} -m}{m} \right)^{n-1}\ell^{w(D)}(\tau^{\prime}_n\circ \epsilon_n)(\sigma),
\end{equation}
where $\epsilon_n$ is the homomorphism from  $B_n$ to ${\mathcal C}_n$ such that  $\sigma_i\mapsto G_i$. Formula (\ref{kauffmanpolynomial}) is deduced by combining (23.1) and (24) of \cite{biweTAMS}, cf.  \cite[Subsection 9.6]{joCBMS}.

Now, by regarding the defining relations of the BMW--algebra and   relations (\ref{b1}), (\ref{b2}), (\ref{bmw4}) and  (\ref{bmw5}), we deduce that there exists an epimorphism from a specialization of the t--BMW algebra in the BMW--algebra; further, this epimorphism factorize   $\varpi$ by  $\tau^{\prime}$. More precisely, we obtain the following proposition, that  can  be  easily  verified.

\begin{proposition}\label{tracetBMWwithBMW}
Setting $\ell = a$,  $m=z$   and $x=\frac{z}{a + a^{-1}-z }$,  we have that the mappings $g_i\mapsto G_i$, $e_i\mapsto 1$ and $h_i, f_i\mapsto H_i$ define an  epimorphism, denoted by $\psi_n$, of ${\Bbb K}$--algebras  from $\mathcal{K}_n(x ,z,a)$     to $\mathcal{C}_n(\ell, m)$. Moreover, we have
$
\tau_n^{\prime} \circ \psi_n = \varpi_n.
$
\end{proposition}

 Under the hypothesis of Proposition \ref{tracetBMWwithBMW} we have that the following  diagram is commutative:

 \begin{equation}\label{commutativediagram2}
\xymatrix{
EB_n \ar[r]^{\tilde{\epsilon}_n} \ar[d]_{\f_n} & \K_n \ar[d]^{\psi_n} \ar[dr]^{\varpi_n} &  \\
B_n \ar[r]_{\epsilon_n} & \mathcal{C}_n  \ar[r]_{\tau'_n}  & \Bbb K  }
\end{equation} where $\tilde{\epsilon}_n$ is the homomorphism from $EB_n$ to
$\mathcal{K}_n$ defined by mapping $\eta^n\sigma_i\mapsto e^ng_i$, and
  $e^n:= e_1\cdots e_{n-1}$.


\begin{proposition}\label{kauffanasL}
The Kauffman polynomial  $\widehat L$ can be obtained as a specialization  of the polynomial  $\widehat \L$.
More precisely, for  an oriented tied--link $\overrightarrow{D}$ whose
components are all tied together, we have
$$\widehat{\L}(\overrightarrow{D} )= \widehat{L}(\overrightarrow{D}^{\nsim}).$$
\end{proposition}
\begin{proof}
    Theorem \ref{TH2} says  that $\widehat{\L}(\overrightarrow{D}) = a^{-w(D)}\L(D)$.   Let  $\eta^n\sigma\in EB_n$  whose closure is $\overrightarrow{D}$; then from (\ref{phi}) we have
$$
\widehat{\L}(\overrightarrow{D})
= a^{-w(D)}x^{1-n}\varpi(\tilde{\epsilon}_n (\eta^n\sigma)).
$$
Now, from the diagram above, we have:
\begin{eqnarray*}
\varpi(\tilde{\epsilon}_n(\eta^n\sigma))& = & (\varpi\circ \tilde{\epsilon_n})(\eta^n\sigma)\\
& = &
((\tau_n^{\prime}\circ \psi_n) \circ \tilde{\epsilon}_n) (\eta^n\sigma)\\
& = &
(\tau_n^{\prime}\circ (\psi_n \circ \tilde{\epsilon}_n)) (\eta^n\sigma)\\
& = &
(\tau_n^{\prime}\circ \epsilon_n \circ \f_n) (\eta^n\sigma)\\
& = &
(\tau_n^{\prime}\circ \epsilon_n) (\f_n (\eta^n\sigma)) = (\tau_n^{\prime}\circ \epsilon_n) (\sigma ).
\end{eqnarray*}
  Recall  that  $w(D)= {\rm exp}(\overrightarrow{D})$ and    $x=\frac{z}{a + a^{-1}-z }$.  Therefore:
$$
\widehat{\L}(\overrightarrow{D}) = a^{-w(D)}x^{1-n}\L(D)= a^{-{\rm exp}(\sigma)}\left( \frac{a + a^{-1}-z }{z}\right)^{n-1}
(\tau_n^{\prime}\circ \epsilon_n) (\sigma ) = \widehat{L}(\overrightarrow{D}^\nsim ),
$$
since     $a=\ell$ and $z= m$.
\end{proof}

\section{ Some  key examples}\label{Sec71}

We  give  here  an  example  of  a  pair of  unoriented  links distinguished  by  $\L$  and $\brla$.

Let L11n304  and L11n412 be  the   unoriented   link  diagrams  obtained  by  forgetting the orientation  in  the  link  diagrams shown in  Fig. \ref{pair1a} (see \cite{link}).
The  calculation  of  $\L$  on  these diagram  gives:

\smallbreak
$\L(L11n304)- \L(L11n412)= -z (-4 z^2 x^2 a^5+z x a^4-3 z^2 x a^3+z^3 x a^4+z^3 x a^2+z^2 x a+a^3 z^4 x-4 a^5 z^2 x
-a^6 z^3 x+a^5 z^4 x-x a^5-x a^3+a^4 z-5 z^2 x^2 a^3+z^4 x^2 a^3+z^4 x^2 a^5-z^3 x^2 a^4 -a^8 z^3 x-2 a^6 z^3 x^2+4 a^6 z x^2+7 z x^2 a^4+2 z x^2 a^2-a^8 z^3 x^2+2 a^7 z^2 x^2
-z x^2+a^9 z^2 x^2)/(x^2 a^5)$
\smallbreak
By  the  variable  change  $a:=-A^3$  and  $z:=A-1/A$   we  get:
\smallbreak
$\bla L11n304 \bra - \bla L11n412 \bra=
-(x^2+2 A^4 x+12 x^2 A^{22}+7 x^2 A^{20}+A^{30} x-A^{16}+2 x^2 A^2-5 x A^{10}-6 x A^{16}+7 A^{20} x-2 A^{14}+4 A^{28} x+6 A^{26} x+5 A^{24} x-7 x^2 A^{14}-7 x^2 A^{12}-6 x^2 A^{10}-4 x^2 A^8-A^{12}+x A^2-5 x A^{14}-3 A^8 x+3 x^2 A^{18}-2 x^2 A^{16}-5 x A^{12}+7 x^2 A^{28}+9 x^2 A^{26}+12 x^2 A^{24}+4 x^2 A^{30}+x^2 A^{32}-x^2 A^6+x^2 A^4+6 A^{22} x)/(A^{17} x^2).$

 \begin{figure}[H]
 \centering{
 \includegraphics[scale=0.9]{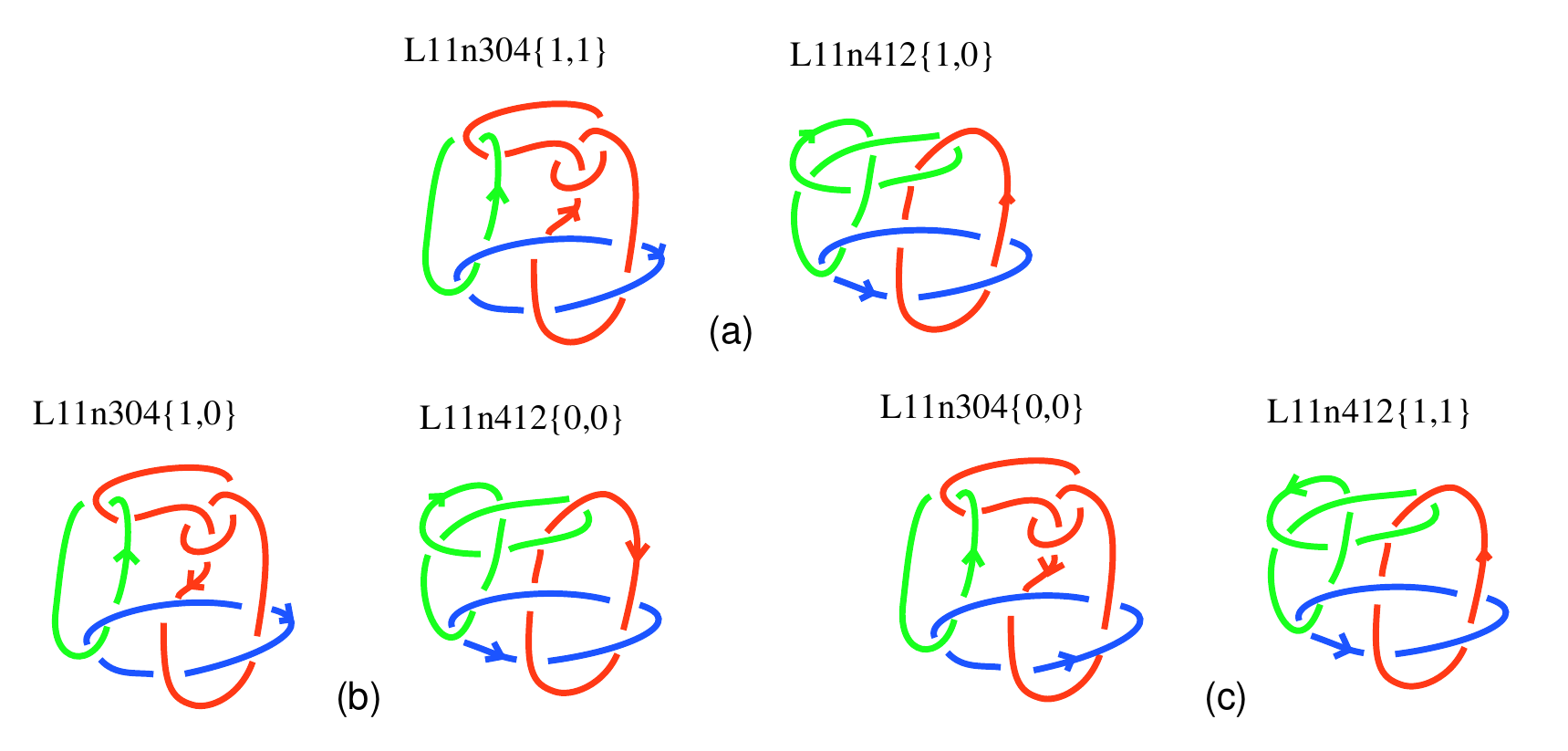}}
 \caption{ }\label{pair1a}
 \end{figure}

   Observe  that  the    oriented  links  of the  pair:
( L11n304\{1,1\}, \    L11n412\{1,0\}),  shown  in  Fig. \ref{pair1a}(a),  have  the  same  Homflypt polynomial,  and  Kauffman  polynomial (see \cite{link}).
The  same  is  true for  the  pairs    (L11n304 \{1,0\}, \ L11n412\{0,0\})  and   (L11n304 \{0,0\}, \  L11n412\{1,1\}),  shown  in  Fig.   \ref{pair1a}(b) and (c).

The  first pair  is  not  distinguished   neither by $\Theta$  nor  by $\F$,  the  others  are  not distinguished by $\Theta$ (see \cite{chljukala}).

  Since  the  writhes  of  the  oriented  diagrams in  each  pair  coincide, it  follows  from  the  calculations  above  that these  oriented  pairs  are    distinguished  by  the  generalized  Kauffman  polynomials   (\ref{Lor}) and (\ref{J}).

\vfill
\newpage

\section*{Appendix}\label{app}

The  following    tables   provide   the  value  of  $\L$  on  simple unoriented  tied  links diagrams,  shown  in  Fig. \ref{list}.
\[ \begin{array} {|l|l| }
   \hline
       &  \L  \\  \hline
   \LL1 & 1 \\   \hline
    \LL2 & (a+1/a)/z-1 \\ \hline
   \LL3 &    1/x\\ \hline
   \LL4 &   ( -a^2-1+a z+z^2 a^2+z^2 ) /( a z ) \\ \hline
   \LL5 &   (-a+z x a^2+z x)/(x a) \\ \hline
    \LL6 &    (-a^2-1+a z)^2/ ( z^2 a^2 ) \\ \hline
   \LL7 &    ( a^2+1-a  z)/(a z x)  \\ \hline
   \LL8 &    1/ x^2   \\ \hline
   \LL9 &    (a^2+1-a z-z^2 a^2-z^2)(-a^2-1+a  z) /(z^2 a^2)\\ \hline
    \LL10 &   ( -a^2-1+a z+z^2 a^2+z^2)/(a z x )  \\  \hline
    \LL11  &     (-a+zx  a^2+z x)(a^2+1-a  z)/(z x a^2)  \\ \hline
   \LL12 &   (  -a+z x a^2+zx)/(x^2 a)  \\  \hline
    \LL13  &  (a^4+2 a^2-2a^3 z+1-2 a z-3 z^2  a^2-2 a^4 z^2+2 z^3 a^3-2z^2+ \\
    &   +2z^3 a+z^4 a^4+2z^4 a^2+z^4)/(z^2 a^2) \\ \hline
    \LL14 &  (a^3+a-a^2 z-z^2  a^3-a z^2-z x  a^4-2 zx a^2+z^2x a^3+z^3x a^4+ \\
    &   +2 z^3x a^2-zx+z^2x a+z^3x)/(z x a^2)
   \\  \hline
    \LL15  &    (a^3+a-a^2 z-2 z x  a^4-4 z x  a^2+2 z^2 x  a^3-2 z x + \\
   &  +2 z^2 x  a+z^3 x  a^4+2 z^3 x  a^2+z^3 x )/(z x  a^2)
      \\ \hline
    \LL16 &    (a^2-2 z x  a^3-2 a z x +z^2 x ^2 a^4+2 z^2 x ^2 a^2+z^2 x ^2) /(x ^2 a^2) \\  \hline
     \LL17^+ &   (a^4+a^2-a^3 z-3 a^4 z^2-2 z^2 a^2+z^3 a^3+z^4 a^4+z^4 a^2+z^3 a+z^2)/(a^3 z) \\  \hline
   \LL17^- &    (a^2+1-a z-2 z^2 a^2-3 z^2+z^3 a+z^4 a^2+z^4+z^3 a^3+a^4 z^2)/(a z)  \\  \hline
   \LL18^+ &  (a^3-3 z x a^4-2 z x a^2+z^2 x a^3+z^3 x a^4+z^3 x a^2+z^2 x a+z x)/(x a^3)   \\  \hline
   \LL18^- & (a-2 z x a^2-3 z x+z^2 x a+z^3 x a^2+z^3 x+z^2 x a^3+z x a^4)/(x a)   \\  \hline
   \LL19^+ &  (-2 a^3-a+a^2 z+z^2 a^3+a z^2+z)/(a^2)   \\  \hline
   \LL19^- &   (-2-a^2+a z+z^2 a^2+z^2+a^3 z)/a  \\  \hline
   \LL20 &  (-1-a^3 z-a z+2 z^2 a^2+z^3 a^3+z^3 a+z^2-a^4-a^2+a^4 z^2)/(a^2)  \\  \hline
   \LL21^+ &    (a^3-a^2 z-4 z^2 a^3-a^4 z^3+3 z^4 a^3+z^5 a^4+z^5 a^2+2 z^4 a+ \\            &   +z^3+a-2 a^5 z^2-2 a z^2+z^4 a^5)/(z a^3)  \\  \hline
   \LL21^-&    (a^2-a^3 z-4 z^2 a^2-z^3 a+3 z^4 a^2+a^3 z^5+a z^5+2 z^4 a^4+ \\       &        +a^5 z^3+a^4-2 a^4 z^2-2 z^2+z^4)/(z a^2)  \\  \hline
   \LL22^+ &    (-z^2 x a^4+3 z^3 x a^3+z^4 x a^4+z^4 x a^2+2 z^3 x a+z^2 x+ \\      &       +a^2-2 z x a^5-4 z x a^3-2 a z x+z^3 x a^5)/(x a^3)  \\  \hline
   \LL22^- &    (-z^2 x a+3 z^3 x a^2+a^3 z^4 x+a z^4 x+2 z^3 x a^4+a^5 z^2 x+\\     &  +  a^3-2 z x a^4-4 z x a^2-2 z x+z^3 x)/(x a^2)  \\  \hline
   \end{array} \]
  \vfill \newpage
\vfill \newpage

\[ \begin{array} {|l|l| }
   \hline
       &  \L  \\  \hline
\LL23^+ & (-1+a^4+a^2-2 a^3 z-2 a^4 z^2-a^2 z^2+2 a^3 z^3+z^4 a^4+a^2 z^4+a z^3+z^2+\\
& -2 a^5 z+a^5 z^3)/(a^3) \\  \hline
     \LL23^- &  (-a^5+a^3+a-2 a^2 z-a^3 z^2-2 a z^2+2 a^2 z^3+a^3 z^4+a z^4+a^4 z^3+a^5 z^2+\\
     & -2 z+z^3)/(a^2) \\  \hline
   \LL24^+ &  (3 a^5+2 a^3-2 a^4 z-4 a^5 z^2-3 a^3 z^2-a^2 z+z+a^4 z^3+z^4 a^5+a^3 z^4+a^2 z^3+\\
   &+a z^2)/(a^4)  \\  \hline
   \LL24^- &   (3+2 a^2-2 a z-3 a^2 z^2-4 z^2-a^3 z+a^5 z+a z^3+a^2 z^4+z^4+a^3 z^3+a^4 z^2)/a \\  \hline
   \LL25^+ &  (-a^6-a^4+a^5 z+6 a^6 z^2+4 a^4 z^2-3 a^5 z^3-5 a^6 z^4-4 z^4 a^4-2 a^3 z^3+\\
   & +a^2 z^2+z^2+a z^3+z^5 a^5+z^6 a^6+z^6 a^4+a^3 z^5+a^2 z^4)/(a^5 z)  \\  \hline
   \LL25^- &  (-a^2-1+a z+4 a^2 z^2+6 z^2-3 a z^3-4 a^2 z^4-5 z^4-2 a^3 z^3-a^4 z^2+\\
   &  +a^6 z^2+a^5 z^3+a z^5+a^2 z^6+z^6+a^3 z^5+z^4 a^4)/(a z)  \\  \hline
   \LL26^+ &    (-a^5+6 a^6 z x+4 z x a^4-3 a^5 z^2 x-5 a^6 z^3 x-4 z^3 x a^4-2 z^2 x a^3-z x a^2+\\
  & +z x+z^2 x a+a^5 z^4 x+z^5 x a^6+z^5 x a^4+a^3 z^4 x+z^3 x a^2)/(x a^5)\\  \hline
   \LL26^- & (-a+4 z x a^2+6 z x-3 z^2 x a-4 z^3 x a^2-5 z^3 x-2 z^2 x a^3-z x a^4+\\
   & +a^6 z x+a^5 z^2 x+a z^4 x+z^5 x a^2+z^5 x+a^3 z^4 x+z^3 x a^4)/(x a) \\  \hline
   \LL27^+ &   (-1+2 a^5 z+2 a^3 z-4 a^4 z^2-3 a^5 z^3-2 a^3 z^3+z^2+2 z^4 a^4+z^5 a^5+\\
   & +a^3 z^5+a^2 z^4+a z^3+a^6+a^4-3 a^6 z^2+a^6 z^4)/(a^4)  \\  \hline
   \LL27^-&    (-a^6+2 a z+2 a^3 z-4 a^2 z^2-2 a^3 z^3-3 a z^3+a^6 z^2+2 a^2 z^4+\\
   & +a^3 z^5+a z^5+z^4 a^4+a^5 z^3+a^2+1-3 z^2+z^4)/(a^2)  \\  \hline
   \LL28^+ &  (-1+3 a^4 z^2+3 a^7 z+3 a^5 z+a z^3+a^2 z^4+z^2-a^4-a^6+4 a^6 z^2-a^3 z^3-3 z^4 a^4+\\
   & +a^3 z^5-6 a^5 z^3-4 a^6 z^4-4 a^7 z^3+a^7 z^5+2 z^5 a^5+z^6 a^6+z^6 a^4)/(a^5)  \\  \hline
   \LL28^- &     (-a^7-a+3 z+z^5-a^3-4 z^3+3 a^3 z^2+4 a z^2+3 a^2 z+z^4 a^5-a^4 z^3+\\
   & -6 a^2 z^3-3 a^3 z^4+z^5 a^4+2 a^2 z^5-4 a z^4+a^7 z^2+a^6 z^3+a^3 z^6+a z^6)/(a^2)  \\  \hline
   \end{array} \]

 \begin{figure}[H]
 \centering \includegraphics[scale=0.8]{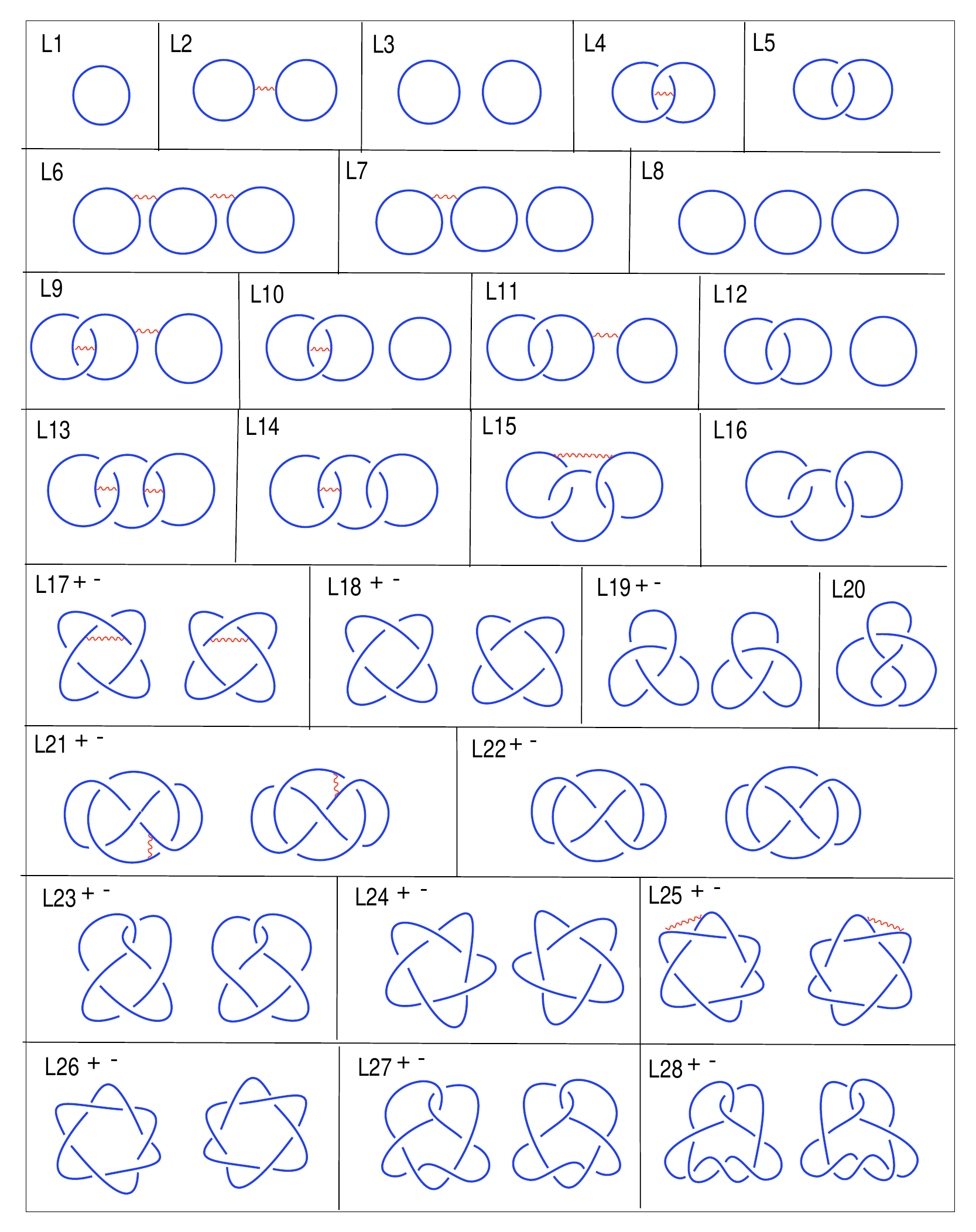}
 \caption{ }\label{list}
 \end{figure}


\begin{thebibliography}{25}

\bibitem{aijuICTP1} F. Aicardi, J. Juyumaya, {\it An algebra involving braids and ties},
ICTP Preprint IC/2000/179, see also arxiv:1709.03740


\bibitem{aijuMMJ1} F. Aicardi, J. Juyumaya, {\it Markov trace on the algebra of braid and ties}, Moscow
Math. J. 16(3) (2016) 397--–431.

\bibitem{aijuJKTR1} F. Aicardi, J. Juyumaya, {\it Tied   Links},   J. Knot Theory Ramifications, \textbf{ 25} (2016), no. 9,   DOI: 10.1142/S02182165164100171.


\bibitem{Aic2} F. Aicardi, {\it New invariants of  links  from a  skein invariant of  colored links}, see  arXiv:1512.00686


\bibitem{biweTAMS} J.S. Birman, and H. Wenzl, {\it Braids, links polynomials and a new algebra}, Trans. Amer. Math. Soc., \textbf{313} (1989), no. 1, 249--273.

\bibitem{link}J. C. Cha, C. Livingston, LinkInfo: Table of Knot Invariants, http://www.indiana.edu/~linkinfo, April 16, 2015.


\bibitem{coetalJA} A.M. Cohen  et al., {\em BMW algebras of simply laced type}, J. Algebra \textbf{286} (2005) 107--153.


\bibitem{chljukala}  M. Chlouveraki et al.,
{\it Identifying  the invariants  for  classical  knots  and links
from  the Yokonuma--Hecke algebras}. http://arxiv.org/pdf/1505.06666.













\bibitem{joAM} V.F.R. Jones, {\it Hecke algebra representations of braid groups and link polynomials},
Ann. Math. \textbf{  126} (1987), 335--388.

   \bibitem{joCBMS} V.F.R. Jones, {\it Subfactors and Knots}, CBMS Regional Conference Series in Mathematics, Amer. Math. Soc.
  \textbf{80}  (1991), 113 pp.









\bibitem{Lic} W.B.R. Lickorish,   {\it An  Introduction  to Knot Theory,}  Graduate  texts in  Mathematics 175, Springer (1991).

\bibitem{mowa}  H. Morton and A. Wassermann,{\it A basis for the Birman--Wenzl algebra}, unpublished manuscript, 1989, revised 2000, 29 pp.

\bibitem{muOJM} J. Murakami,  {\it The Kauffman polynomial of links and representation theory}, Osaka J. Math., \textbf{24} (1987), 745--758.

\bibitem{kaTOP}  L. Kauffman, {\it States model and the Jones polynomial}, Topology, \textbf{26} (1987), no. 3, 395--407.

\bibitem{kaTAMS} L. Kauffman, {\it An invariant of regular  isotopy}, Trans. Am. Math. Soc., \textbf{318} (1990), no.2, 417--471.





\end{thebibliography}
\end{document}